\documentclass[12pt,twoside]{article}
\usepackage[left=2.5cm,right=2.5cm]{geometry}
\usepackage{amsmath,amsthm,amscd,amssymb,amsfonts}
\usepackage{latexsym,graphics,longtable,lscape}
\usepackage{mathtools, mathrsfs}	
\usepackage{epstopdf}
\usepackage{float}
\usepackage{indentfirst}
\usepackage{cite}
\usepackage{verbatim}
\usepackage{mleftright}
\usepackage{enumitem}
\renewcommand{\baselinestretch}{1}

\usepackage{color}
\definecolor{myblue}{rgb}{0.13,0.29,0.56}

\usepackage{varioref}
\usepackage{nameref}
\usepackage[CJKbookmarks=true,colorlinks,linkcolor=black,anchorcolor=black,citecolor=black]{hyperref}
\usepackage[nameinlink]{cleveref}

\crefname{section}{\color{black} Section}{\color{black} Section}
\crefname{lemma}{Lemma}{Lemmas}
\crefname{prob}{Problem}{Problems}
\crefname{coro}{Corollary}{Corollaries}
\crefname{conj}{Conjecture}{Conjectures}
\crefname{prop}{Proposition}{Propositions}
\crefname{claim}{Claim}{Claims}
\crefname{defn}{Definition}{Definitions}
\crefname{rmk}{Remark}{Remarks}
\crefname{exm}{Example}{Examples}
\crefname{thm}{Theorem}{Theorems}

\def\br{\mathbb{R}}
\def\bn{\mathbb{N}}

\theoremstyle{plain}
\newtheorem{thm}{Theorem}[section]
\newtheorem{lemma}[thm]{Lemma}

\newtheorem{coro}[thm]{Corollary}

\theoremstyle{definition}

\newtheorem{rmk}[thm]{Remark}

\newcommand{\bremark}{\begin{remark} \em}
\newcommand{\eremark}{\end{remark} }

\numberwithin{equation}{section}

\begin{document}
\parindent 15pt
\renewcommand{\theequation}{\thesection.\arabic{equation}}
\renewcommand{\baselinestretch}{1.15}
\renewcommand{\arraystretch}{1.1}
\renewcommand{\vec}[1]{\bm{#1}}
\def\disp{\displaystyle}
\title{\bf\large Symmetry of Convex Solutions to Fully Nonlinear Elliptic Systems: Unbounded Domains\thanks{Supported by National Natural Science Foundation of China (11771428, ????????)}
\author{{\small Weijun Zhang$^{1}$\thanks{Email: zhangweij3@mail.sysu.edu.cn}~~~Zhitao Zhang$^{2,3,4}$\thanks{Corresponding author. Email: zzt@math.ac.cn}}\\
{\small $^1$ School of Mathematics, Sun Yat-sen University, Guangzhou 510275, People's Republic of China,}\\
{\small $^2$ Academy of Mathematics and Systems Science,}\\
{\small Chinese Academy of Sciences, Beijing 100190, People's Republic of China,}\\
{\small $^3$ School of Mathematical Sciences, University of Chinese}\\
{\small Academy of Sciences, Beijing 100049, People's Republic of China,}\\
{\small $^4$ School of Mathematical Science, Jiangsu University, Zhenjiang 212013, People's Republic of China.}}
}

\date{}

\maketitle

\abstract
{\small In this paper, we are concerned with the monotonic and symmetric properties of convex solutions Monge-Amp\`ere systems for instance, considering
	\begin{equation*}
		\det(D^2u^i)=f^i(x,{\bf u},\nabla u^i), \ 1\leq i\leq m,
	\end{equation*}
over unbounded domains of various cases, including the whole spaces $\br^n$, the half spaces $\br^n_+$ and the unbounded tube shape domains in $\br^n$. We obtain monotonic and symmetric properties of the solutions to the problem with respect to the geometry of domains and the monotonic and symmetric properties of right-hand side terms. The proof is based on carefully using the moving plane method together with various maximum principles and Hopf's lemmas.
}

\section{Introduction}
\noindent

In this paper, we consider the following Monge-Amp\`ere systems:
\begin{equation}\label{eq:MA}
\det(D^2 u^i(x))=f^i(x,{\bf u}(x),\nabla u^i(x)),\ in\ \Omega,\ 1\leq i\leq m,
\end{equation}
where $\Omega\subset\br^n$, ${\bf u}=(u^1,\dots,u^m)$, and ${\bf f}=(f^1,\dots,f^m)$, $n,m\in\bn^*$, satisfy some suitable conditions in different cases.

\subsection{Background}

The monotonicity and symmetry of solutions to nonlinear partial differential equations has many applications in mathematics, such as ensure the uniqueness of solution of some nonlinear elliptic equations as we see in \cite{zhang2024}; derive the a-priori bound or the behaviors at the infinity for solutions to nonlinear elliptic equations; discover the bifurcation phenomenon\cite{zhang_existence_2009,zhang_power-type_2015}, especially the situations when symmetry breaking; determine the geometry of a manifolds \cite{alexandrov_characteristic_1962} and so on. See \cite{chen_methods_2010, chen_moving_2003, zhang_variational_2013} and the reference therein for more examples. 

These properties have been studied in many years by many mathematician, see \cite{gidas_symmetry_1979,gidas_symmetry_1981,li_monotonicity_1991,li_monotonicity_1991-1} and the reference therein. However, to the best of our knowledge, there are few papers concerning as for Monge-Amp\`ere equations, especially the case of unbounded domains, except \cite{ma_symmetry_2010-1,cui_symmetry_2019}.

Our principal goal of this paper is to give a rather complete and general version of monotonic and symmetric results to the Monge-Amp\`ere system over unbounded domains of various cases, including the whole spaces $\br^n$, the half spaces $\br^n_+$ and the unbounded tube shape domains in $\br^n$. The cases of bounded domains can be seen here \cite{zhang2024}.

\subsection{Main Results}
The main results about symmetry are formally stated as below, in fact we get a more general results about monotonicity, more detail could be seen in \cref{sec:thewhole,sec:thehalf,sec:ubt}.

In order to state our main results, we need firstly introduce some basic hypotheses on $f^i:\overline{\Omega}\times\br^m\times\br^n\to\br, (x,{\bf z},p)\to f^i(x,{\bf z},p)$, where ${\bf z}=(z^1,\dots,z^m)$. We suppose that for all $1\leq i\leq m$, $f^i\in C(\overline{\Omega}\times\br^m\times\br^n,\br)$, furthermore satisfying some of the following in different situations.

In order to assure the ellipticity of \eqref{eq:MA}, we need the following two kinds of positive conditions:
\begin{enumerate}[resume,label=$(F_{\arabic{enumi}})$]
	\item \label{F0} $f^i(x,{\bf z},p)>0$, $\forall (x,{\bf z},p)\in (\Omega\times\br^m\times\br^n)$;
	\item \label{Fc} $f^i(x,{\bf z},p)\geq c_f>0$, $\forall (x,{\bf z},p)\in (\Omega\times\br^m\times\br^n)$;
\end{enumerate}

\begin{rmk}
	\ref{Fc}, which is in order to assure the uniformly ellipticity of \eqref{eq:MA}, is stronger than \ref{F0}, which is merely assure the ellipticity of \eqref{eq:MA}.
\end{rmk}

Next, when $\Omega$ assume to be convex in one direction, denote as ${\bf e_1}$, we can study whether the solutions to \eqref{eq:MA} will be having some monotonicity, hence we need the following monotonicity kind conditions on ${\bf f}$:

\begin{enumerate}[resume,label=$(F_{\arabic{enumi}})$]

\item \label{Fziinf} $f^i$ is non-decreasing in $z^i$, whenever the remaining components $z^j$, $j\neq i$, and $x,p$ fixed;
	\item \label{Fzj}$f^i$ is non-increasing in $z^j$, $j\neq i$, whenever the remaining components $z^k$, $k\neq j$, and $x,p$ fixed;
	\item \label{Fp}$f^i$ is locally uniformly Lipschitz continuous in the component $p$, whenever $x,{\bf z}$ fixed;
        \item \label{Fanti}$f^i(y_1,x',{\bf z},\bar{p})\geq f^i(x,{\bf z},p),\ \forall\ {\bf z}\in \br^m, p\in\br^n$ and $x=(x_1,x')\in\Omega$ such that $p_1\leq 0, x_1\leq 0$ with $x_1\leq y_1 \leq -x_1$, where $\bar{p}:=(-p_1,p_2,\cdots,p_n)$;
\end{enumerate}

In order to determine the consistency of the symmetric center of each $u^i$, we sometimes use the following conditions instead of \ref{Fzj}.
\begin{enumerate}[resume,label=$(F_{\arabic{enumi}})$]
\item \label{Ffc} $\forall I,J\subset\{1,2,\dots,m\}, I,J\neq\emptyset, I\cap J=\emptyset, I\cup J=\{1,2,\dots,m\},\ \exists i_0\in I, j_0\in J,$ denote the set of $x\in\Omega$ as $A_f^{i_0,j_0}$ such that $f^{i_0}$ is strictly decreasing with respect to $z^{j_0}$ whenever the remaining components $z^k, k\neq i_0,j_0,$ and $p$ fixed, we have $|A_f^{i_0,j_0}|>0$.
\end{enumerate}

\begin{rmk}
    Condition \ref{Fzj} is called the weak coupled condition or cooperative condition, and condition \ref{Ffc} is called the strong coupled condition, both are insuring the solutions $u^i$ to have the same monotonicity or symmetry, while the latter one can also determine the consistency of the symmetric center.
\end{rmk}

Furthermore, when $\Omega$ assume to be symmetric in ${\bf e_1}$, we can study whether the solutions to \eqref{eq:MA} will be having some symmetry along ${\bf e_1}$, we need to strengthen \ref{Fanti}.

\begin{enumerate}[resume,label=$(F_{\arabic{enumi}})$]
	\item \label{Fsym1}$f^i(x,{\bf z},p)= f^i(|x_1|,x_2,\dots,x_n,{\bf z},|p_1|,p_2,\dots,p_n),\ \forall (x,{\bf z},p)\in (\Omega\times\br^m\times\br^n)$;
\end{enumerate}

At last, when $\Omega$ assume to be symmetric in all directions, we can study whether the solutions to \cref{eq:MA} will be radially symmetry, we need to strengthen \ref{Fsym1}.

\begin{enumerate}[resume,label=$(F_{\arabic{enumi}})$]
	\item \label{Fsymall} $f^i(x,{\bf z},p)=f^i(Ox,{\bf z},O'p),\ \forall\ O, O’\in O(n),\ \forall (x,{\bf z},p)\in(\Omega\times\br^m\times\br^n)$, where $O(n)$ is the n-th order orthogonal group;
	
\end{enumerate}

For the convenience, we denote
\begin{equation}\label{eq:dijinf}
\widetilde{d_{ij}}(x,{\bf z},p,h):=\left\{\begin{array}{rl}
\frac{1}{h}\left(f^i(x,{\bf z}+h{\bf e_j},p)-f^i(x,{\bf z},p)\right),& h\neq 0,\\
0,& h=0,\\
\end{array}
\right.
\end{equation}
and $\widetilde{D}:=\begin{pmatrix}
\widetilde{d_{11}}&\cdots&\widetilde{d_{1m}}\\
\vdots&\ddots&\vdots\\
\widetilde{d_{m1}}&\cdots&\widetilde{d_{mm}}
\end{pmatrix},$

\begin{rmk}\label{rmk:dijinf}
 By \ref{Fziinf} and \ref{Fzj}, we have $\operatorname{sgn}(\widetilde{d_{ij}}\cdot h)=-\operatorname{sgn}(h)$, and  $\widetilde{d_{ij}}\leq0, \ \forall\ i,j=1,\dots,m,\ i\neq j$, while $\widetilde{d_{ii}}\geq0,\ \forall\ i=1,\dots,m$.
\end{rmk}

Denote $\widetilde{D_i}$ be the $i$-th ordered minor of $\widetilde{D}$, we need to require some positivity on it at the infinity:
\begin{enumerate}[resume, label=$(F_{\arabic{enumi}})$]

	\item \label{FH_i} 
    $\widetilde{D_h^i}:=\lim\limits_{|x|^{-1}+\sum\limits_{j=1}^m|z^j|^{-1}\to 0}\inf\limits_{h\in\br}\widetilde{D_i}(x,z^1,\cdots,z^m,0,h)>0$.
\end{enumerate}

\begin{rmk}
 Condition \ref{FH_i} can be used to describe the behaviors of $u^i$ at the infinity. In particular, if $\widetilde{D}$ is symmetric, then by Sylvester's criteria, condition \ref{FH_i} is equivalent to: $$\lim\limits_{|x|^{-1}+\sum\limits_{j=1}^m|z^j|^{-1}\to 0}\inf\limits_{h\in\br}\widetilde{D}(x,z^1,\cdots,z^m,0,h)\ \text{is positive definite}.$$
\end{rmk}

\vskip 0.1in
Now we begin to state our main results about symmetry.

\vskip 0.1in
For the case of $\Omega=\br^n$, we mainly consider the entire problem,
\begin{equation}\label{eq:wholeintro}
\left\{\begin{array}{rl}
\det(D^2 u^i)\ =&f^i(x,{\bf u},\nabla u^i),\quad\ \text{in}\ \br^n, \\
\lim\limits_{|x|\to +\infty}u^i(x)\ =&\infty,\quad\ \quad\ \quad\ \quad\ \ 1\leq i\leq m, \\

\end{array}
\right.
\end{equation}

We have the main results as follow,
\begin{thm}
	Assume ${\bf f}$ satisfy \ref{F0}, \ref{Fziinf}, \ref{Fp}, \ref{Ffc}, \ref{Fsymall} and \ref{FH_i}. Let ${\bf u}=(u^1,\cdots,u^m)$ be a group of $[C^2(\br^n)]^m$ strictly convex solutions of \eqref{eq:wholeintro} satisfying some growth conditions as in \cref{sec:thewhole}, then each $u^i$ must be radially symmetric and strictly increasing respect to the some point in $\br^n$. 
	
	More precisely, denote the rotating center as $x^*\in\br^n$, and $r=|x-x^*|$, then for $i=1\dots,m$, each $u^i$ must be like
	\begin{equation*}
	u^i(x)=u^i(r),\ \forall\ x\in\br^n,
	\end{equation*}
	moreover,
	\begin{equation*}
	\frac{d u^i}{dr}(x)>0,\ \forall\ x\in\br^n,
	\end{equation*}
\end{thm}

\vskip 0.1in
For the case of $\Omega=\br^n_+$, we consider the Neumann problem in the half space. Denote $x=(x',x_n)$, $x'=(x_1,\cdots,x_{n-1})\in \br^{n-1}$, $x_n\geq 0$, and $r=|x'|$, for all $1\leq i\leq m$, 

\begin{equation}\label{eq:halfNintro}
\left\{\begin{array}{rl}
&\det(D^2 u^i)\ =f^i(x,{\bf u},\nabla u^i),\ x\in\ \br^n_+, \\
&u^i(x)>0,\hspace{2.9cm}\ x\in\ \br^n_+,\\
&\frac{\partial u^i}{\partial x_n}(x)=h^i(r),\hspace{2.05cm}\ x\in\ \partial\br^n_+,\\
&\lim\limits_{|x|\to\infty}u^i(x)\ =\infty,
\end{array}
\right.
\end{equation}
We have the main results as follow,
\begin{thm}
	Assume ${\bf f}$ satisfy \ref{Fc}, \ref{Fp}, \ref{Ffc}, \ref{Fsymall}, \ref{FH_i} and \ref{Fziinfh},  ${\bf h}$ satisfy \ref{H}. Let ${\bf u}=(u^1,\cdots,u^m)$ be a group of $[C^2(\overline{\br^n_+},\br_+)]^m$ strictly convex solutions to \eqref{eq:halfNintro} satisfying some growth conditions as in \cref{sec:thewhole}, then each $u^i$ must be radially symmetric and strictly increasing respect to some point in $\br^n_+$. 
	
	More precisely, denote the rotating center as $x^*=((x^*)',x^*_n)\in\br^n_+$, and $r=|x'-(x^*)'|$, then for $i=1\dots,m$, each $u^i$ must be like
	\begin{equation*}
		u^i(x)=u^i(r,x_n),\ \forall\ x\in \br^n_+,
	\end{equation*}
	moreover, 
	\begin{equation*}
		\frac{\partial u^i}{\partial r}(r,x_n)>0,\ \forall\ x\in \br^n_+.
	\end{equation*}
\end{thm}

\vskip 0.1in
For the case of unbounded tubes, 
we mainly consider the following constant-boundary Dirichlet problem for \eqref{eq:MA} on $C_{\infty}$,
\begin{equation}\label{eq:cylinderinfintro}
\left\{\begin{array}{rl}
\det(D^2 u^i)\ =&f^i(x,{\bf u},\nabla u^i), x\in C_{\infty},\\
u^i\ =&h^i,\hspace{1.8cm}x\in \partial C_{\infty},\\
\lim\limits_{|x|\to +\infty}u^i(x)\ =&\infty, \hspace{1.8cm} i=1,\dots,m.\\
\end{array}
\right.
\end{equation}
where $h^i$ satisfy $\lim\limits_{|x_n|\to +\infty}h^i(x)=\infty.$

We have the main results as follow,
\begin{thm}
	Let $C_{\infty}=\Omega\times(-\infty, \infty)$ be a infinite cylinder in $\br^n$, $\Omega=B_R$ be a arbitrary ball with radius $R$ in $\br^{n-1}$. Assume ${\bf f}$ satisfy \ref{F0}, \ref{Fziinf}, \ref{Fzj}, \ref{Fp}, \ref{Fsymall} and \ref{FH_i}. Let ${\bf u}=\left(u^1,\cdots,u^m\right)\in \left[C^2\left(\overline{C_{\infty}}\right)\right]^m$ be a group of strictly convex solutions to \eqref{eq:cylinderinfintro} satisfying some growth conditions as in \cref{sec:thewhole,sec:ubt}, then each $u^i$ must be radially symmetric and strictly increasing respect to the axis crossing the center of $B_R$.
	
	More precisely, denote the center of $B_R$ as $x^*=\left((x^*)',x^*_n\right)\in\br^n$, and denote $r=|x'-(x^*)'|$, then for $i=1\dots,m$, each $u^i$ must be
	\begin{equation*}
	u^i(x)=u^i(r,x_n),\ x\in C_{\infty},
	\end{equation*}
	moreover, 
	\begin{equation*}
	\frac{\partial u^i}{\partial r}(r,x_n)>0,\ x\in C_{\infty}.
	\end{equation*}
\end{thm}

We mainly follow the moving plane method with concrete procedures proposed by Troy\cite{troy_symmetry_1981} and Busca \cite{busca_symmetry_2000}, and recently developed by Ma-Liu \cite{ma_symmetry_2010-1,liu_symmetry_2012,liu_symmetry_2013}. Various maximum principles and Hopf’s lemmas are repeatedly used in the proof.

With respect to the cases of unbounded domains, we mainly improve the existing result by reducing the smoothness condition on the right-hand side $f^i$ from $C^1$ to Lipschitz continuous. In particular, we simplify the assumptions in \cite{ma_symmetry_2010-1} for the case of whole spaces, while our results for the case of half spaces and of the unbounded tube shape domains are new. The method is spiritually similar to \cite{santos_symmetry_2020}, where symmetry properties were obtained for positive solutions to certain fully nonlinear elliptic systems mainly dominated by Pucci operators.

This paper is organized as follows. In \cref{sec:pre}, we present some preliminary results for the moving plane method. \cref{sec:thewhole,sec:thehalf,sec:ubt} are concerned with the cases of unbounded tubes, the whole spaces and the half spaces, respectively. More specifically, in \cref{sec:thewhole}, we study the entire problem of Monge-Amp\`ere systems in the whole space; in \cref{sec:thehalf}, we deal with the Neumann boundary problem of Monge-Amp\`ere systems in the half space; and \cref{sec:ubt} is devoted to the case of unbounded tube shape domains.  

\vskip 0.2in
\section{Some Preliminaries}\label{sec:pre}
\noindent

Note that, in what follows, we always consider the classical solutions to the problem, that is, the solutions being twice continuously differentiable up to the boundary. This is always the case if each $f^i(x,{\bf u}(x),\nabla u^i(x))$ is $C^\alpha(\overline{\Omega})$ as a function of $x$ by the standard regularity theory of Monge-Amp\`ere equation; see \cite{figalli_monge-ampere_2017,gilbarg_elliptic_2001,le_schauder_2017,cheng_regularity_1977}. And in order to assure the ellipticity of the equations, the solutions are always considered to be strictly convex.

Here are some notations preparing for the moving plane method.
Fixed a direction vector $\nu\in\br^n$ with $|\nu|=1$, and a real number $\lambda\in\br$, we defined the related half space
\begin{equation*}
	\Sigma_{\lambda,\nu}:=\{x\in\Omega\ |\ x\cdot\nu<\lambda\},
\end{equation*}
and the corresponding hyperplane
\begin{equation*}
	T_{\lambda,\nu}:=\{x\in\Omega\ |\ x\cdot\nu=\lambda\}.
\end{equation*}

Let $x_{\lambda,\nu}$ be the reflection of $x\in\overline{\Omega}$ through $T_{\lambda,\nu}$, that is 
\begin{equation*}
	x_{\lambda,\nu}:=x+2(\lambda-x\cdot\nu)\nu.
\end{equation*}

correspondingly, for any set $A\subset\br^n$, let $A^\nu_\lambda$ be the reflection through $T_{\lambda,\nu}$, that is
\begin{equation*}
	A^\nu_\lambda:=\{x_{\lambda,\nu}=x+2(\lambda-x\cdot\nu)\nu\ |\ x\in A\}.
\end{equation*}

We denote that for a invertible matrix $M$, $M^{jk}:=(M^{-1})_{jk}$, and for two matrices $M_1,M_2$, denoted the Frobenius inner product as  $$\langle M_1,M_2\rangle_F:=\sum\limits_{j,k=1}^n(M_1)_{jk}(M_2)_{jk}=\operatorname{tr}\left(M_1^TM_2\right),$$
especially, if one of them is symmetric, then $\langle M_1,M_2\rangle_F=\operatorname{tr}\left(M_1M_2\right)$.

For a function $u\in C^2(\overline{\Omega})$, we define the reflected function $u_{\lambda,\nu}(x)$ through $T_{\lambda,\nu}$ as follow,
\begin{equation*}
u_{\lambda,\nu}(x):=u(x_{\lambda,\nu})=u(x+2(\lambda-x\cdot\nu)\nu),
\end{equation*}
and we have
\begin{equation*}
\frac{\partial u_{\lambda,\nu}}{\partial x_j}(x)=\sum\limits_{i=1}^n\frac{\partial u}{\partial x_i}(x_{\lambda,\nu})(\delta_{ij}-2\nu_i\nu_j)=\nabla u(x_{\lambda,\nu})\cdot \mu^{\nu}_j,
\end{equation*}
where $\mu^{\nu}_j:=(-2\nu_1\nu_j,\cdots,1-2\nu_j^2,\cdots,-2\nu_n\nu_j),$
and then
\begin{equation*}
\frac{\partial^2 u_{\lambda,\nu}}{\partial x_k\partial x_j}(x)=\nabla (\frac{\partial u_{\lambda,\nu}}{\partial x_j}(x))\cdot \mu^{\nu}_j=\nabla (\nabla u(x_{\lambda,\nu})\cdot \mu^{\nu}_k)\cdot \mu^{\nu}_j=\langle D^2(x_{\lambda,\nu}),(\mu^{\nu}_k)^T\mu^{\nu}_j\rangle_F,
\end{equation*}
thus
\begin{equation*}
\nabla u_{\lambda,\nu}(x)=\nabla u(x_{\lambda,\nu})\cdot
\begin{pmatrix}
1-2\nu_1^2 & -2\nu_1\nu_2 & \cdots & -2\nu_1\nu_n \\
-2\nu_2\nu_1 & 1-2\nu_2^2 & \cdots & -2\nu_2\nu_n \\
\vdots & \vdots & \ddots & \vdots \\
-2\nu_n\nu_1 & -2\nu_n\nu_2 & \cdots & 1-2\nu_n^2
\end{pmatrix}
=\nabla u(x_{\lambda,\nu})\cdot(I-2\nu^T\nu),
\end{equation*}

And we at last define the difference function $U_\lambda(x)$
\begin{equation*}
U_{\lambda,\nu}(x):=u_{\lambda,\nu}(x)-u(x).
\end{equation*}

Once if the domain is somehow convex in one direction, for example $\nu=e_1=(1,0,\dots,0)\in\br^n$, in this case, for shortly, we denote
\begin{align*}
	&T_\lambda:=T_{\lambda,e_1}=\{x\in\Omega\ |\ x_1=\lambda\},\\ &\Sigma_\lambda:=\Sigma_{\lambda,e_1}=\{x\in\Omega\ |\ x_1<\lambda\},\\ 
	&x_\lambda:=x_{\lambda,e_1}=(2\lambda-x_1,x'),\ where\  x'=(x_2,\dots,x_n)\in\br^{n-1}\\ &u_\lambda(x):=u_{\lambda,e_1}(x)=u(2\lambda-x_1,x'). 
\end{align*} 

We can easily see that $$\nabla u_\lambda(x)=\left(-\frac{\partial u}{\partial x_1}(x_\lambda),\frac{\partial u}{\partial x_2}(x_\lambda),\dots,\frac{\partial u}{\partial x_n}(x_\lambda)\right)=\nabla u(x_\lambda)\cdot \bar{D},$$ and the Hessian matrix of $u_\lambda$ is
\begin{equation*}
D^2u_\lambda(x)=
\begin{pmatrix}
	\frac{\partial^2 u}{\partial x_1^2}(x_\lambda) & -\frac{\partial^2 u}{\partial x_1\partial x_2}(x_\lambda) & \cdots & -\frac{\partial^2 u}{\partial x_1\partial x_n}(x_\lambda) \\
	-\frac{\partial^2 u}{\partial x_2\partial x_1}(x_\lambda) & \frac{\partial^2 u}{\partial x_2^2}(x_\lambda) & \cdots & \frac{\partial^2 u}{\partial x_2\partial x_n}(x_\lambda) \\
	\vdots & \vdots & \ddots & \vdots \\
	-\frac{\partial^2 u}{\partial x_n\partial x_1}(x_\lambda) & \frac{\partial^2 u}{\partial x_n\partial x_2}(x_\lambda) & \cdots & \frac{\partial^2 u}{\partial x_n\partial x_n}(x_\lambda)
\end{pmatrix}
=\bar{D}^TD^2u_(x_\lambda)\bar{D},
\end{equation*}
where $\bar{D}=\operatorname{diag}\{-1,1,\cdots,1\}$. Note that $|\nabla u_\lambda(x)|=|\nabla u(x_\lambda)|$ and the eigenvalue of $D^2u_\lambda(x)$ are the same as $D^2u(x_\lambda)$, especially,
\begin{equation}\label{eq:det}
	\det(D^2u_\lambda(x))=\det(D^2u(x_\lambda)).
\end{equation}

And we define the difference function in direction $x_1$, $$U_\lambda(x):=U_{\lambda,e_1}(x)=u_\lambda(x)-u(x).$$

Note that at the special case $x=x_\lambda$, that is, $x\in\overline{T_\lambda}$, we have the following useful results:
\begin{equation}\label{nablaonT}
	\nabla U_\lambda(x)=\left(-2\frac{\partial u}{\partial x_1}(x),0,\dots,0\right);
\end{equation}
\begin{equation}\label{HessainonT}
D^2U_\lambda(x)=
\begin{pmatrix}
0 & -2\frac{\partial^2 u}{\partial x_1\partial x_2}(x) & \cdots & -2\frac{\partial^2 u}{\partial x_1\partial x_n}(x) \\
-2\frac{\partial^2 u}{\partial x_2\partial x_1}(x) & 0 & \cdots & 0 \\
\vdots & \vdots & \ddots & \vdots \\
-2\frac{\partial^2 u}{\partial x_n\partial x_1}(x) & 0 & \cdots & 0
\end{pmatrix}
.
\end{equation}

\vskip 0.2in

Now we are ready to do some preliminary calculations for \eqref{eq:MA}. Firstly, we have \begin{equation}\label{eq:ddet}
	\frac{\partial}{\partial q_{ij}}\det(M)=\det(M)M^{ij}, ~~~\forall M\text{ being positive definite},
\end{equation}
then by the integral form of mean value theorem, we have

 \begin{equation}\label{eq:mv}
 \det\left(D^2u^i_\lambda(x)\right)-\det\left(D^2u^i(x)\right)=\langle {\bf A^i}(x),D^2U^i_\Lambda(x)\rangle_F=\operatorname{tr}\left({\bf A^i}(x)D^2U^i_\lambda(x)\right),
 \end{equation}
 where ${\bf A^i}(x):=(a^i_{jk}(x))_{j,k=1}^n$ with
 \begin{equation}\label{eq:aij}
 a^i_{jk}(x):=\int_0^1\det\left((1-t)D^2u^i_\lambda(x)+tD^2u^i(x)\right)\left((1-t)D^2u^i_\lambda(x)+tD^2u^i(x)\right)^{jk}dt.
 \end{equation}
 
  On the other hand,  $\forall\lambda<0$, $x\in\Sigma_\lambda$, we have $x_1<(x_\lambda)_1<-x_1$, hence by \eqref{eq:det} and \ref{Fanti}, $\forall x\in\Sigma_\lambda$ such that $\frac{\partial u^i}{\partial x_1}(x)\leq0$, we have
 \begin{equation}\label{eq:anti}
 \begin{split}
\det(D^2u^i_\lambda(x))&=\det(D^2u^i(x_\lambda))\\
 &=f^i(x_\lambda,{\bf u}(x_\lambda),\nabla u^i(x_\lambda))\\
 &=f^i(x_\lambda,{\bf u}_\lambda(x),\overline{\nabla u^i_\lambda(x)})\\
 &\geq f^i(x,{\bf u}_\lambda(x),\nabla u^i_\lambda(x)),
 \end{split}
 \end{equation}
 hence by \ref{Fziinf},\ref{Fzj} and \ref{Fp}, we have
 \begin{equation}\label{eq:splitinf}
 \begin{split}
 &\det(D^2u^i_\lambda)-\det(D^2u^i)\\
 \geq& f^i(x,{\bf u}_\lambda,\nabla u^i_\lambda)-f^i(x,{\bf u},\nabla u^i)\\
 =&f^i(x,{\bf u}_\lambda,\nabla u^i_\lambda)-f^i(x,{\bf u},\nabla u^i_\lambda)+f^i(x,{\bf u},\nabla u^i_\lambda)-f^i(x,{\bf u},\nabla u^i)\\
 =&f^i(x,{\bf u}_\lambda,\nabla u^i_\lambda)-f^i(x,u^1_\lambda,\cdots,u^{m-1}_\lambda,u^m,\nabla u^i_\lambda)+\cdots\\
 &+f^i(x,u^1_\lambda,\cdots,u^i_\lambda,\cdots,u^m,\nabla u^i_\lambda)-f^i(x,u^1_\lambda,\cdots,u^i,\cdots,u^m,\nabla u^i_\lambda)+\cdots\\
 &+f^i(x,u^1_\lambda,u^2,\cdots,u^m,\nabla u^i_\lambda)-f^i(x,{\bf u},\nabla u^i_\lambda)\\
 &+f^i(x,{\bf u},\nabla u^i_\lambda)-f^i(x,{\bf u},\nabla u^i)\\
 \geq& \widetilde{d_{im}}(x,u^1_\lambda,\cdots,u^{m-1}_\lambda,u^m,\nabla u^i_\lambda,U^m_\lambda)U^m_\lambda+\cdots\\
&+\widetilde{d_{ii}}(x,u^1_\lambda,\cdots,u^i,\cdots,u^m,\nabla u^i_\lambda,U^i_\lambda)U^i_\lambda+\cdots\\
&+\widetilde{d_{i1}}(x,{\bf u},\nabla u^i_\lambda,U^1_\lambda)U^1_\lambda-h_{f^i,p}|\nabla U^i_\lambda|,
 \end{split}
 \end{equation}
where $d_{ij}$ are defined as \eqref{eq:dijinf}, and $h_{f^i,p}$ is the Lipschitz constants of $f^i$ in \ref{Fp}.
 
 Then combining \eqref{eq:mv} and \eqref{eq:splitinf} we can obtain an elliptic inequality of $U^i_\lambda$ in $\Sigma_\lambda$:
 \begin{equation}\label{eq:EIinf}
\begin{split}
\operatorname{tr}\left({\bf A^i}(x)D^2U^i_\lambda(x)\right)+&{\bf B^i}(x)\cdot\nabla U^i_\lambda(x)\\
\geq&\sum_{j=1}^m \widetilde{d_{ij}}(x,u^1_\lambda,\cdots,u^j,\cdots,u^m,\nabla U^i_\lambda,U^j_\lambda)U^j_\lambda,
\end{split}
\end{equation}
 with ${\bf A^i}(x)$ defined as \eqref{eq:aij}, and 
\begin{equation}\label{eq:bi}
{\bf B^i(x)}:=\frac{h_{f^i,p}}{|\nabla U^i_\lambda(x))|}\chi_{\{|\nabla U^i_\lambda(x))|\neq0\}}\nabla U^i_\lambda(x),
\end{equation}

\vskip 0.2in

Next we give some lemmas here for convenience.
In the procedure of using moving plane methods, the following strong maximum principle and Hopf's Lemma will be crucial. The proof of it can be found in \cite{gilbarg_elliptic_2001}.
\begin{lemma}[Maximum Principle $\&$ Hopf's Lemma]\label{lem:SMP-HL}
	Let $\Omega\in\br^n$ be a domain, $w\in C^2(\Omega)$ be a non-positive solution in $\Omega$ to the following elliptic inequality
	\begin{equation*}
		\operatorname{tr}\left(A(x)D^2w(x)\right)+{\bf B}(x)\cdot\nabla w(x)+c(x)w(x)\geq0,
	\end{equation*}
	where $A(x):=(a_{ij}(x))_{i,j=1}^n$, ${\bf B}(x):=(b_i(x))_{i=1}^n$, and $a_{ij},b_i,c\in L_{\operatorname{loc}}^\infty(\Omega)$ with $A(x)$ is locally positive definite in $\Omega$. Then either $w\equiv 0$ or $w<0$ in $\Omega$.
	
	Moreover, if $w(x_0)>0$	for some $x_0\in\Omega$, and $w(\bar{x})=0$ for some $\bar{x}\in\partial\Omega$, near which $w$ is continuously differentiable, then	
	$$\frac{\partial w}{\partial \nu}(\bar{x})>0,$$	
	where $\nu$ is the unit outer normal of $\partial\Omega$.
\end{lemma}
\vskip 0.2in

In our case, since the domain we dealing with may not satisfy the interior ball condition, we will use the boundary point Hopf lemma at a corner instead, which is the content of the following lemma in \cite{gidas_symmetry_1979} (due to Serrin \cite{serrin_symmetry_1971}).
\begin{lemma}[Serrin's Corner Lemma]\label{lemmaS}
	Let $\Omega$ be a domain in $\br^n$ with the origin $Q$ on its boundary. Assume that near $Q$ the boundary consists of two transversally intersecting $C^2$ hypersurfaces $\{\rho=0\}$ and $\{\sigma=0\}$. Suppose $\rho,\sigma<0$ in $\Omega$. Let $w$ be a function in $C^2(\overline{\Omega})$, with $w<0$ in $\Omega$, $w(Q)=0$, satisfying the differential inequality
	\begin{equation*}
	a_{ij}w_{x_ix_j}+b_i(x)w_{x_i}+c(x)w\geq 0\text{ in }\Omega,
	\end{equation*}
	with uniformly bounded coefficients satisfying $a_{ij}\xi_i\xi_j\geq c_0|\xi|^2$. Assume 
	\begin{equation}\label{eq:SL}
	a_{ij}\rho_{x_i}\sigma_{x_j}\geq 0\text{ at }Q.
	\end{equation}
	If this is zero, assume furthermore that $a_{ij}\in C^2$ in $\overline{\Omega}$ near $Q$, and that 
	\begin{equation*}
	D(a_{ij}\rho_{x_i}\sigma_{x_j})= 0\text{ at }Q,
	\end{equation*}
	for any first order derivative $D$ at $Q$ tangent to the submanifold $\{\rho=0\}\cap\{\sigma=0\}$. Then, for any direction $s$ at $Q$ which enters $\Omega$ transversally to each hypersurface,
	\begin{align*}
	&\frac{\partial w}{\partial s}<0\text{ at $Q$ in case of strict inequality in \eqref{eq:SL},}\\
	&\frac{\partial w}{\partial s}<0\text{ or }\frac{\partial^2 w}{\partial s^2}<0\text{ at $Q$ in case of equality in \eqref{eq:SL}.}
	\end{align*}
\end{lemma}

In order to overcome the difficulties coming from coupling systems, we need the following lemma of linear algebra.
\begin{lemma}\label{lemmaA}
Let $M=(m_{ij})_{i,j=1}^n$ be a real matrix, satisfying $m_{ij}\leq0, i\neq j$. Assume all the principal ordered minors of $M$ are positive definite, denote its adjoined matrix as $\operatorname{adj}(M)$, then
		\begin{enumerate}[label=(\roman{enumi})]
				\item all the same order minors of $M$ are positive definite,
				\item all the algebraic remainders of $M$ are non-negative definite, that is $\operatorname{adj}(M)_{ij}\geq0$.
		\end{enumerate}
\end{lemma}

In order to overcome the difficulties coming from unboundedness of domains, we need two more conditions on $u^i$ at the infinity. For $u^i\in C^1(\br^n)$, we denote the radial derivative as $$\frac{\partial u^i}{\partial r}(x):=\langle\nabla u^i(x), \frac{x}{|x|}\rangle,$$ the tangential derivative as $$\nabla_{\tau}(x):=\nabla u^i(x)-\frac{\partial u^i}{\partial r}(x)\frac{x}{|x|},$$ We need the following condition inspired by \cite{porretta_symmetry_2006}:
\begin{align*}
	&\lim\limits_{|x|\to\infty}\frac{\partial u^i}{\partial r}(x)>0,\\
	&|\nabla_{\tau}(x)|=o\left(\frac{\partial u^i}{\partial r}(x)\right)\ \text{as}\ |x|\to\infty.\stepcounter{equation}\tag{\theequation}\label{eq:inf1}
\end{align*}

Recall that $x_{0,\nu}=-2(x\cdot\nu)\nu+x$, we also require that
\begin{align*}
	&\lim\limits_{x\cdot\nu\to-\infty}u^i(x_{0,\nu})-u^i(x)\leq0,\\
	\text{equivalently,}\  &\lim\limits_{x\cdot\nu\to-\infty}\frac{u^i(x_{0,\nu})}{u^i(x)}\leq1.\stepcounter{equation}\tag{\theequation}\label{eq:inf2}
\end{align*}

\vskip 0.2in

\section{The Whole Space $\Omega=\br^n$.}\label{sec:thewhole}
\noindent

\begin{figure}[ht]
	\centering
	\includegraphics[width=2.5in]{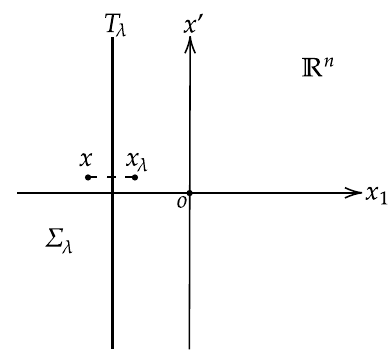}
	\caption{The whole space}
	\label{fig:whole}
\end{figure} 

In this case, we consider the following entire problem:
\begin{equation}\label{eq:whole}
\left\{\begin{array}{rl}
\det(D^2 u^i)\ =&f^i(x,{\bf u},\nabla u^i),\quad\ \text{in}\ \br^n, \\
\lim\limits_{|x|\to +\infty}u^i(x)\ =&\infty,\quad\ \quad\ \quad\ \quad\ \ 1\leq i\leq m, \\
\end{array}
\right.
\end{equation}

The behavior of $u^i$ at the infinity as above are natural since we always assume $u^i$ to be strictly convex.

\vskip 0.2in
\subsection{Main Theorem}

We begin to state our main theorems.

When $\bf f$ being monotonic in one direction, assumed as $\bf e_1$, we can start to examine whether the solution to the system \eqref {eq:whole} will satisfy the corresponding monotonicity along this direction. The main results are the following.

\begin{thm}\label{thm:whole}
	Assume ${\bf f}$ satisfy \ref{F0}, \ref{Fziinf}, \ref{Fzj}, \ref{Fp}, \ref{Fanti} and \ref{FH_i}. Let ${\bf u}=(u^1,\cdots,u^m)$ be a group of $[C^2(\br^n)]^m$ strictly convex solutions of \eqref{eq:whole} satisfying \eqref{eq:inf1} and \eqref{eq:inf2}, then there exists $t_1\leq0$ such that for each $u^i$,
	$$u^i(x_1,x')\geq u^i(2t_1-x_1,x')\ \text{and}\ \frac{\partial u^i}{\partial x_1}(x)<0\ \forall\ x\in \br^n\ \text{with}\ x_1<t_1.$$
	Furthermore, if $\frac{\partial u^i}{\partial x_1}(t_1,x')=0$ for some $x\in \{x_1=t_1\}$, with $t_1=0$ (or $t_1<0$), then such (or all) $u^i$ must be symmetric with respect to  $\{x_1=t_1\}$ and strictly decreasing in $x_1$ direction with $x_1<t_1$,
	that is,
	\begin{equation*}
	u^i(x)=u^i(|x_1-t_1|,x'),\ \forall\ x\in \br^n,
	\end{equation*}
	moreover, 
	\begin{equation*}
	\frac{\partial u^i}{\partial x_1}<0,\ \forall x\in\br^n\ \text{with}\ \ x_1<t_1.
	\end{equation*}
\end{thm}

\vskip 0.2in

If we assume more symmetry on ${\bf f}$ (substituting \ref{Fanti} with \ref{Fsym1}), we can furthermore immediately have the following, by using \cref{thm:whole} again with ${\bf u_{t_1}}:=(u^i_{t_1})$. (Note that in this case, the inequalities \eqref{eq:anti},\eqref{eq:splitinf} will be slightly different to obtain the same result.)

\begin{thm}
	Assume ${\bf f}$ satisfy \ref{F0}, \ref{Fziinf}, \ref{Fp}, \ref{Ffc}, \ref{Fsym1} and \ref{FH_i}. Let ${\bf u}=(u^1,\cdots,u^m)$ be a group of $[C^2(\br^n)]^m$ strictly convex solutions of \eqref{eq:whole} satisfying \eqref{eq:inf1} and \eqref{eq:inf2}, then there exists $t_1\in\br$ such that for each $u^i$ must be symmetric with respect to $\{x_1=t_1\}$ and strictly decreasing in $x_1$ direction with $x_1<t_1$. 
	
	More precisely, for all $i=1,\dots, m$, each $u^i$ must be like
	\begin{equation*}
	u^i(x)=u^i(|x_1-t_1|,x'),\ \forall\ x\in \br^n,
	\end{equation*}
	moreover, 
	\begin{equation*}
	\frac{\partial u^i}{\partial x_1}<0,\ \forall x\in\br^n\ \text{with}\ \ x_1<t_1.
	\end{equation*}
\end{thm}

\vskip 0.2in

Especially, if we substitute \ref{Fanti} to the symmetric one \ref{Fsymall}, then by using \cref{thm:whole} with respect to all directions in $\br^n$, we have:

\begin{coro}\label{thm:wholefinal}
	Assume ${\bf f}$ satisfy \ref{F0}, \ref{Fziinf}, \ref{Fp}, \ref{Ffc}, \ref{Fsymall} and \ref{FH_i}. Let ${\bf u}=(u^1,\cdots,u^m)$ be a group of $[C^2(\br^n)]^m$ strictly convex solutions of \eqref{eq:whole} satisfying \eqref{eq:inf1} and \eqref{eq:inf2}, then each $u^i$ must be radially symmetric and strictly increasing respect to the some point in $\br^n$. 
	
	More precisely, denote the rotating center as $x^*\in\br^n$, and $r=|x-x^*|$, then for $i=1\dots,m$, each $u^i$ must be like
	\begin{equation*}
	u^i(x)=u^i(r),\ \forall\ x\in\br^n,
	\end{equation*}
	moreover,
	\begin{equation*}
	\frac{d u^i}{dr}(x)>0,\ \forall\ x\in\br^n,
	\end{equation*}
\end{coro}

\vskip 0.2in
\subsection{Proof of \cref{thm:whole}}

We are now in a position to prove the theorem.

By the same procedure in calculating \eqref{eq:EIinf}, we can obtain an elliptic inequality of $U^i_\lambda$ in $\Sigma_\lambda$:
$\forall x\in\Sigma_\lambda$ such that $\frac{\partial u^i}{\partial x_1}(x)\leq0$, we have
\begin{equation}\label{eq:EIW}
\begin{split}
\operatorname{tr}\left({\bf A^i}(x)D^2U^i_\lambda(x)\right)+&{\bf B^i}(x)\cdot\nabla U^i_\lambda(x)\\
\geq&\sum_{j=1}^m \widetilde{d_{ij}}(x,u^1_\lambda,\cdots,u^j,\cdots,u^m,\nabla U^i_\lambda,U^j_\lambda)U^j_\lambda,
\end{split}
\end{equation}
where ${\bf A^i}(x):=(a^i_{jk}(x))_{j,k=1}^n$ with
\begin{equation*}
a^i_{jk}(x):=\int_0^1\det\left((1-t)D^2u^i_\lambda(x)+tD^2u^i(x)\right)\left((1-t)D^2u^i_\lambda(x)+tD^2u^i(x)\right)^{jk}dt,
\end{equation*}
are strictly positive definite due to the strictly convexity of $u^i$ and \ref{F0}, and together with $${\bf B^i(x)}:=\frac{h_{f^i,p}}{|\nabla U^i_\lambda(x))|}\chi_{\{|\nabla U^i_\lambda(x))|\neq0\}}\nabla U^i_\lambda(x)$$ 
are all locally bounded due to the twice differentiable continuity of $u^i$.

\vskip 0.2in

We firstly prove some lemmas.

We prove the same strong maximum principle as in \cite{zhang2024}.
\begin{lemma}\label{lemma:SMPFinf}
	Assume ${\bf f}$ satisfy \ref{F0}, \ref{Fziinf}, \ref{Fzj}, \ref{Fp}, \ref{Fanti}, let ${\bf u}$ be $[C^2(\Omega)]^m$ strictly convex solutions to \eqref{eq:whole}, if $$U^i_\lambda\leq0, \frac{\partial u^i}{\partial x_1}\leq0,\ \forall x\in\Sigma_\lambda,\ \forall i=1,\dots,m,$$ 
	then either $$U^i_\lambda<0,\ \forall x\in\Sigma_\lambda,\ \forall \ i=1,\dots,m,$$
	or $$U^i_\lambda\equiv0,\ \forall x\in\Sigma_\lambda,\ \forall \ i=1,\dots,m.$$
\end{lemma}

Next, we describe the behaviors of $U^i_\lambda$ at the infinity
\begin{lemma}\label{lemma:inf1}
	Let ${\bf u}=(u^1,\cdots,u^m)$ satisfy \eqref{eq:inf1} and \eqref{eq:inf2}, then for all $\lambda<0$, 
	$$\lim\limits_{|x|\to\infty,x_1<\lambda}U^i_\lambda(x)\leq0, \forall i=1,\dots,m.$$ 
\end{lemma}
\begin{proof}
	By \eqref{eq:inf1},
	\begin{align*}
	\frac{\partial u^i}{\partial x_1}&=\frac{\partial u^i}{\partial r}\frac{x_1}{|x|}+\left(\nabla u^i-\frac{\partial u^i}{\partial r}\frac{x}{|x|}\right)\cdot{\bf e_1}\\
	&=\frac{\partial u^i}{\partial r}\left(\frac{x_1}{|x|}+o(1)\right),\ \text{as}\ |x|\to\infty,\ \forall\ i=1,\dots,m,
	\end{align*}
	Hence for all $\lambda<0$,
	\begin{align*} U^i_\lambda(x)&=u^i(2\lambda-x_1,x')-u^i(x)\\
	&=u^i(2\lambda-x_1,x')-u^i(-x_1,x')+u^i(-x_1,x')-u^i(x)\\
	&=2\lambda\int_0^1\frac{\partial u^i}{\partial x_1}\left(2(1-t)\lambda-x_1,x'\right)dt+u^i(-x_1,x')-u^i(x)\\
	&=2\lambda\int_0^1\frac{\partial u^i}{\partial r}\left(\frac{2(1-t)\lambda-x_1}{|x|}+o(1)\right)dt+u^i(-x_1,x')-u^i(x),\ \text{as}\ |x|\to\infty,
	\end{align*}
	Noticing that $\frac{\partial u^i}{\partial r}>0$, as $|x|\to\infty$, hence for $x_1$ sufficiently small such that $x_1<2\lambda$, then by \eqref{eq:inf2}, $$\lim\limits_{|x|\to\infty,x_1<2\lambda}U^i_\lambda(x)\leq0.$$ 
	On the other hand, for $\lambda<0$,
	\begin{align*} U^i_\lambda(x)&=u^i(2\lambda-x_1,x')-u^i(x_1,x')\\
		&=2(\lambda-x_1)\int_0^1\frac{\partial u^i}{\partial x_1}\left(2t(\lambda-x_1)+x_1,x'\right)dt\\
		&=2(\lambda-x_1)\int_0^1\frac{\partial u^i}{\partial r}\left(\frac{2t(\lambda-x_1)+x_1}{|x|}+o(1)\right)dt,\ \text{as}\ |x|\to\infty,
	\end{align*}
	Noticing that $\frac{\partial u^i}{\partial r}>0$,as $|x|\to\infty$, hence for $x_1$ sufficiently small such that $2\lambda<x_1<\lambda$, we have
	$$\lim\limits_{|x|\to\infty,2\lambda<x_1<\lambda}U^i_\lambda(x)\leq0.$$ 
\end{proof}

\begin{rmk}
This lemma is in order to ensure the super-mum of $U_\lambda^i$ on $\Sigma_\lambda$ can be achieved inside $\Sigma_\lambda$, not just a sequence of maximizers. This version is the axis symmetric case, in fact we can weaken the condition \eqref{eq:inf1} and \eqref{eq:inf2} if we only consider one direction.
\end{rmk}

\begin{lemma}\label{lemma:inf2}
	Let ${\bf u}=(u^1,\cdots,u^m)$ be a group of strictly convex solutions to \eqref{eq:MA}, ${\bf f}$ satisfy \ref{FH_i}, then exists $\widetilde{R}>0$ such that for all $\lambda\in\br$, $|x|>\widetilde{R}$, $i=1,\dots,n$,
	\begin{align*}
	\widetilde{D_i}(x,z^1,\cdots,z^m,0,h)>0, \forall h\in\br
	\end{align*}
	where $z^k$ lies in the segment between $u^k_\lambda(x)$ and $u^k(x)$, $k=1,\dots,m$.
\end{lemma}
\begin{proof}
	Noting that $\lim\limits_{|x|\to\infty}u^i(x)\ =\infty$, and then by \ref{FH_i}.
\end{proof}

\vskip 0.2in

We are now on the position of proving the main theorem.

\newenvironment{prf4}{{\noindent\bf Proof of \cref{thm:whole}.}}{\hfill $\square$\par}
\begin{prf4}
	We begin to move the hyperplane parallel to $T_\lambda$ coming from the $-\infty$.

 {\bf Step 1}: There exists a real number $\lambda\in\br$, such that $\left.U^i_\mu\right|_{\Sigma_\mu}\leq0, \forall\ i=1,\dots,m, \forall\ \mu<\lambda, $.
	
	If not, suppose that $\forall \lambda<0$, exists $y_0\in\Sigma_\lambda$ such that for some $i_0$, $U^{i_0}_\lambda(y_0)>0$. Denote index set $J:=\{j=1,\dots,m\ |\ U^j_\lambda(x)\leq0,\ \forall x\in\Sigma_\lambda\}\subsetneqq\{1,\dots,m\}$, and $I:=\{1,\dots,m\}\setminus J.$ Without lost of generality, we can assume $I=\{1,\dots,l\},\ J=\{l+1,\dots,m\},\ l=|I|$. 
	
	Noting that $\sup\limits_{\Sigma_\lambda}U^i_\lambda>0$, $\forall\ i\in I$. By \cref{lemma:inf1}, we can choose $\{y_i\}_{i\in I}\subset \overline{\Sigma_\lambda}$ such that
	\begin{equation*}
	U^i_\lambda(y_i)=\max_{y\in\overline{\Sigma_\lambda}}U^i_\lambda(y)>0.
	\end{equation*}
	For fixed $i\in I$, $\left.U^i_\lambda\right|_{T_{\lambda}}=0$ shows that $y_i\in\Sigma_\lambda$, hence $\nabla U^i_\lambda(y_i)=0$, and $D^2 U^i_\lambda(y_i)\leq 0$. Then the $i$-th equation in \eqref{eq:EIW} at $y_i$ transforms to
	\begin{align*}
	0\geq\sum_{j=1}^m \widetilde{d_{ij}}(y_i,u^1_\lambda(y_i),\cdots,u^j(y_i),\cdots,u^m(y_i),0,U^j_\lambda(y_i))U^j_\lambda(y_i),
	\stepcounter{equation}\tag{\theequation}\label{eq1}
	\end{align*}
	By \cref{rmk:dijinf}, definition of $J$ and $U^j_\lambda(y_i)\leq U^j_\lambda(y_j), \forall j\in I$, \eqref{eq1} becomes
	\begin{align*}
	0\geq\sum_{j\in I} \widetilde{d_{ij}}(y_i,u^1_\lambda,\cdots,u^j,\cdots,u^m,0,U^j_\lambda(y_i))U^j_\lambda(y_j).
	\stepcounter{equation}\tag{\theequation}\label{eq2}
	\end{align*}
	
	We will prove that $U^j_\lambda(y_j)\leq0$, $\forall j\in I$, which will be contradictory to the choice of $y_i$. Rewrite \eqref{eq2} as
	\begin{align*}
	MU=V,
	\stepcounter{equation}\tag{\theequation}\label{eq3}
	\end{align*}
	where $U:=\left(U^1_\lambda(y_1),\dots,U^l_\lambda(y_l)\right)$, $V:=(v_1,\dots,v_l)$, $M:=(m_{ij})_{i,j=1}^l$, satisfying $$v_i\leq0, m_{ij}:=\widetilde{d_{ij}}(y_i,u^1_\lambda,\cdots,u^j,\cdots,u^m,0,U^j_\lambda(y_i))\leq0,\ i,j=1,\dots,l,\ i\neq j.$$ 
	We can assume $\lambda<-\widetilde{R}$, where $\widetilde{R}$ defined in \cref{lemma:inf2}, then $\Sigma_\lambda\subset\br^n\setminus B_{\widetilde{R}}(0)$, hence by \cref{lemma:inf2} and \cref{lemmaA}, $M$ is invertable, hence by \eqref{eq3} we can solve $U=M^{-1}V$, by Cramer's law and \cref{lemmaA}, together with \cref{lemma:inf2}, we have
	$$U^j_\lambda(y_j)=M^{jk}v_k=\frac1{\det M}\sum_{k=1}^l\operatorname{adj}(M)_{jk}v_k\leq0,\ \forall j=1,\dots,l.$$\\
	
	Next we continue to move $T_\lambda$ by increasing $\lambda$, as long as $U^i_\lambda(x)\leq 0$ holds for all $i$ on $\Sigma_{\Lambda}$, we define
	\begin{equation*}
		\Lambda:=\sup\mleft\{\lambda\in\br\ \middle|\ U^i_\mu(x)\leq 0,\ \forall x\in\Sigma_\mu,\ \forall\ i=1,\dots,m,\ \forall\ \mu<\lambda\mright\}.
	\end{equation*}
	Step 1 shows that $\Lambda>-\infty$. On the other hand $\lim\limits_{|x|\to +\infty}u^i(x)=\infty$, for sufficiently large $R$, we have $u^i(x)>u^i(0)$, $\forall|x|>R$, hence $\Lambda<+\infty$, therefore $\Lambda$ is finite. 
	Since all the terms are continuous with respect to $\lambda$, we have 
	\begin{equation}\label{eq:wholestep1.5}
	U^i_\Lambda(x)\leq 0,\ \forall x\in\Sigma_\Lambda,\ \forall\ i=1,\dots,m.
	\end{equation}
	In particular, $U^i_\mu\leq 0, \forall \mu\leq\Lambda$, hence $\forall x\in\Sigma_\Lambda$, $\frac{\partial u^i}{\partial x_1}(x)\leq0$ by covering argument. By \eqref{eq:EIW} and \cref{rmk:dijinf}, we have on $\Sigma_\Lambda$,
	\begin{equation}\label{eq:EIWF}
	\operatorname{tr}\left({\bf A^i}(x)D^2U^i_\Lambda(x)\right)+{\bf B^i}(x)\cdot\nabla U^i_\Lambda(x)-\widetilde{d_{ii}}(x)U^i_\Lambda(x)\geq0,\ \forall\ i=1,\dots,m.
	\end{equation}
	
	Since all the coefficient of \eqref{eq:EIWF} are locally bounded, by \cref{lem:SMP-HL}, we have $\forall i=1,\dots,m$, either
	\begin{align*}
	&U^i_\Lambda(x)<0, \hspace{0.1cm}\forall x\in \Sigma_\Lambda,\\
	\text{or}\hspace{0.1cm}&U^i_\Lambda(x)\equiv 0, \hspace{0.1cm}\forall x\in \Sigma_\Lambda.
	\end{align*}
	If the former one happens, we also have
	\begin{align*}
	\frac{\partial U^i_\Lambda}{\partial x_1}(x)>0, \hspace{0.1cm}\forall x\in T_\Lambda.
	\end{align*}
	
	Now, if $\Lambda=0$, we have $\frac{\partial u^i}{\partial x_1}(0,x')=0$ for some $x\in \{x_1=0\}$, hence $U^i_\Lambda(x)\equiv 0, \hspace{0.1cm}\forall x\in \Sigma_\Lambda$.
	
	Next, we discuss the case of $\Lambda<0$.
	
	\textbf{Step 2}: We will prove that $U^i_\Lambda\equiv 0$ on $\Sigma_\Lambda$ for at least one $i$. 
	
	As the discussion above, we only need to exclude the case that all the $U^i_\Lambda$ on $\Sigma_\Lambda$ are strictly negative, with directional derivative along $x_1$ being strictly positive on $T_\Lambda$. Assume this is happening.
	
	By definition of $\Lambda$, there exists a sequence of $\{\lambda_k\}_{k=1}^\infty\subset\br$ and a sequence of $\{y_k\}_{k=1}^\infty\subset\br^n$, such that $\lambda_k\in(\Lambda, 0)$, satisfying $\lim\limits_{k\to\infty}\lambda_k=\Lambda$, and $y_k\in\Sigma_{\lambda_k}$, such that at least for one $i$, $U^i_{\lambda_k}(y_k)>0$. By \cref{lemma:inf1}, we can choose $\{y_k\}\subset\overline{\Sigma_{\lambda_k}}\setminus\Sigma_{\Lambda}$, such that
	\begin{equation}\label{wstep2 max}
	U^i_{\lambda_k}(y_k)=\max_{y\in\overline{\Sigma_{\lambda_k}}}U^i_{\lambda_k}(y)>0, \hspace{0.1cm}k=1,2,\cdots
	\end{equation}
	There two cases will happen. 

\begin{enumerate}
	\renewcommand{\theenumi}{\textbf{Case \arabic{enumi}}}
	\renewcommand{\labelenumi}{\theenumi}
	\item the sequence $\{y_k\}_{k=1}^\infty$ exists bounded subsequence, that is $x_k\to x^*\in\overline{\Sigma_\Lambda}$.
	
	By continuity of $u^i$, $U^i_{\Lambda}(x^*)\geq 0$, Hence by $\left.U^i_\Lambda\right|_{\Sigma_\Lambda}<0$ we have $x^*\in T_\Lambda$. By mean value theorem, $$0\leq U^i_{\Lambda}(x_k)=2\frac{\partial u^i_\Lambda}{\partial x_1}(\xi_k)(\lambda_k-x_{k,1}),$$ where $\xi_k$ lies in the segment connecting $x_k$ and $x_k^{\lambda_k}$, which shows $\frac{\partial u^i_\Lambda}{\partial x_1}(\xi_k)\geq0$, hence $$\frac{\partial u^i_\Lambda}{\partial x_1}(x^*)\geq0,$$ By continuity of $u^i$, this is contradictory to $\left.\frac{\partial u^i}{\partial x_1}\right|_{T_\Lambda}=\left.-\frac12\frac{\partial U^i_\Lambda}{\partial x_1}\right|_{T_\Lambda}<0$. 
	
	\item  $\lim_{k\to\infty}|y_k|=+\infty$.  
	In this case, we can choose $k^*>0$ such that $\forall k>k^*$, $|y_k|>\widetilde{R}$, where $\widetilde{R}$ as in \cref{lemma:inf2}. Then for each $y_k, k>k^*$, the argument similar to \textbf{Step 1} will give a contradiction.
 \end{enumerate}
	
	\textbf{Step 3}: We need to prove all the $U^i_\Lambda\equiv 0$ on $\Sigma_\Lambda$. 
	
	By \textbf{Step 2}, denote $I:=\mleft\{i\ \middle|\ \left.U_\Lambda^i\right|_{\Sigma_\Lambda}\equiv0\mright\}$, $J:=\{1,\dots,m\}\setminus I$. If $J\neq\emptyset$, by \ref{Ffc}, there exists $i_0\in I, j_0\in J,$ such that
	$$\left|\mleft\{x\in\br^n\ \middle|\ f^{i_0} \text{strictly decreasing with respect to} z^{j_0}, \text{as} z^k, k\neq i_0,j_0, p\text{fixed}\mright\}\right|>0,$$ that is there exists $x_0\in\br^n$, such that $$\widetilde{d_{i_0j_0}}(x_0,u^1_\Lambda(x_0),\cdots,u^j(x_0),\cdots,u^m(x_0),0,U^{j_0}_\Lambda(x_0))U^{j_0}_\Lambda(x_0)<0.$$
	
	If $x_0\in\Sigma_\Lambda$, then by \cref{lemma:SMPFinf} and \eqref{eq:EIW}, together with $\widetilde{d_{ij}}\leq 0,i\neq j$ in \cref{rmk:dijinf}, and \eqref{eq:wholestep1.5} show that
	\begin{align*}
		0\geq&\sum\limits_{j\neq i_0} \widetilde{d_{i_0j}}U^j_\Lambda=\sum\limits_{j\in J} \widetilde{d_{i_0j}}U^j_\Lambda\\
		\geq&\widetilde{d_{i_0j_0}}U^{j_0}_\Lambda>0,
	\end{align*}
	which is a contradiction.
	
	If $x_0\in\Sigma_\Lambda^\Lambda$, consider the elliptic inequality of $\Sigma_\Lambda^\Lambda$ on $U^i_\Lambda$ should be
	\begin{equation*}
	\begin{split}
	\operatorname{tr}\left({\bf A^i}(x)D^2U^i_\Lambda(x)\right)+&{\bf B^i}(x)\cdot\nabla U^i_\Lambda(x)\\
	\geq&\sum_{j=1}^m -\widetilde{d_{ij}}(x,u^1_\Lambda,\cdots,u^j,\cdots,u^m,\nabla U^i_\Lambda,U^j_\Lambda)U^j_\Lambda.
	\end{split}
	\end{equation*}
	Noting that for $j\in J$, $U^j_\Lambda$ should be strictly positive on $\Sigma_\Lambda^\Lambda$, hence
	\begin{align*}
	0\geq&\sum\limits_{j\neq i_0} -\widetilde{d_{i_0j}}U^j_\Lambda=\sum\limits_{j\in J} -\widetilde{d_{i_0j}}U^j_\Lambda\\
	\geq& -\widetilde{d_{i_0j_0}}U^{j_0}_\Lambda>0,
	\end{align*}
	by the same reason, which comes out a contradiction again. Hence $J=\emptyset$, that is all the $U^i_\Lambda\equiv 0$ hold on $\Sigma_\Lambda$.
	
\end{prf4}

\vskip 0.2in

\section{The Half Space $\Omega=\br^n_+$.}\label{sec:thehalf}
\noindent

\begin{figure}[ht]
	\centering
	\includegraphics[width=2.7in]{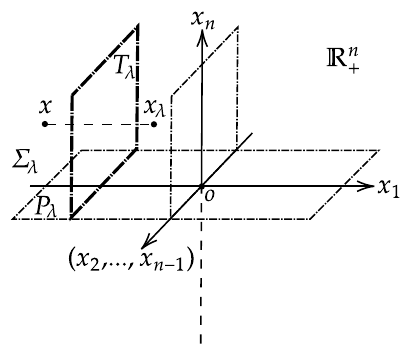}
	\caption{The half space}
	\label{fig:half}
\end{figure} 

In this case, we consider the following Neumann problem in the half space $\Omega=\br^n_+$, $n\geq 2$. Denote $x=(x',x_n)$, $x'=(x_1,\cdots,x_{n-1})\in \br^{n-1}$, $x_n\geq 0$, and $r=|x'|$, for all $1\leq i\leq m$, 

\begin{equation}\label{eq:halfN}
\left\{\begin{array}{rl}
&\det(D^2 u^i)\ =f^i(x,{\bf u},\nabla u^i),\ x\in\ \br^n_+, \\
&u^i(x)>0,\hspace{2.9cm}\ x\in\ \br^n_+,\\
&\frac{\partial u^i}{\partial x_n}(x)=h^i(r),\hspace{2.05cm}\ x\in\ \partial\br^n_+,\\
&\lim\limits_{|x|\to\infty}u^i(x)\ =\infty,
\end{array}
\right.
\end{equation}

We have here a hypothesis on the boundary condition:
\begin{enumerate}[label=$(H)$]
	\item \label{H} $h^i(r)\in C^1([0,+\infty))$, $\frac{\partial h^i}{\partial r}\leq 0$, $i=1,\dots,m$. 
\end{enumerate}

The behavior of $u^i$ at the infinity as above are natural since we always assume $u^i$ to be strictly convex.

We need one more technical condition of ${\bf f}$,
\begin{enumerate}[label=$(F_n)$]
	\item \label{Fziinfh} $f^i$ is non-decreasing with respect to $z^i$ and strictly increasing on $\{x_n=0\}$ such that $\widetilde{d_{ii}}\geq d>0$, as $z^k$, $k\neq i$, $x,p$ fixed;
\end{enumerate}

\vskip 0.2in
\subsection{Main Theorem}
The main results are the following.

\begin{thm}\label{thm:half}
	Assume ${\bf f}$ satisfy \ref{Fc}, \ref{Fzj}, \ref{Fp}, \ref{Fanti}, \ref{FH_i} and \ref{Fziinfh},  ${\bf h}$ satisfy \ref{H}. Let ${\bf u}=(u^1,\cdots,u^m)$ be a group of $[C^2(\overline{\br^n_+},\br_+)]^m$ strictly convex solutions to \eqref{eq:halfN} satisfying \eqref{eq:inf1} and \eqref{eq:inf2}, then there exists $t_1\leq0$ such that for each $u^i$,
	$$u^i(x_1,x'',x_n)\geq u^i(2t_1-x_1,x'',x_n),\ \forall\ x\in \br^n_+\ \text{with}\ x_1<t_1$$
	where $x''=(x_2,\cdots,x_{n-1})\in \br^{n-2}$, and $$\frac{\partial u^i}{\partial x_1}<0,\ \forall\ x\in \br^n_+\ \text{with}\ x_1<t_1.$$
	Furthermore, if $\frac{\partial u^i}{\partial x_1}(t_1,x')=0$ for some $x\in \{x_1=t_1\}$, with $t_1=0$ (or $t_1<0$) then such (or all) $u^i$ must be symmetric with respect to  $\{x_1=t_1\}$ and strictly decreasing in $x_1$ direction with $x_1<t_1$,
	that is,
	\begin{equation*}
		u^i(x)=u^i(|x_1-t_1|,x'',x_n),\ \forall\ x\in \br^n_+,
	\end{equation*}
	moreover, 
	\begin{equation*}
		\frac{\partial u^i}{\partial x_1}<0,\ \forall x\in\{x\in \br^n_+\ |\ x_1<t_1\}.
	\end{equation*}
\end{thm}

\vskip 0.2in

If we assume more symmetry on $f^i$ (substituting \ref{Fanti} with \ref{Fsym1}), we can furthermore immediately have the following, by using \cref{thm:half} again with ${\bf u_{t_1}}:=(u^i_{t_1})$. (Note that in this case, the inequalities \eqref{eq:anti},\eqref{eq:splitinf} will be slightly different to obtain the same result.)

\begin{thm}
	Assume ${\bf f}$ satisfy \ref{Fc}, \ref{Fp}, \ref{Ffc}, \ref{Fsym1}, \ref{FH_i} and \ref{Fziinfh},  ${\bf h}$ satisfy \ref{H}. Let ${\bf u}=(u^1,\cdots,u^m)$ be a group of $[C^2(\overline{\br^n_+},\br_+)]^m$ strictly convex solutions to \eqref{eq:halfN} satisfying \eqref{eq:inf1} and \eqref{eq:inf2}, then there exists $t_1\in\br$ such that for each $u^i$ must be symmetric with respect to  $\{x_1=t_1\}$ and strictly decreasing in $x_1$ direction with $x_1<t_1$. 
	
	More precisely, for all $i=1,\dots, m$, each $u^i$ must be like
	\begin{equation*}
	u^i(x)=u^i(|x_1-t_1|,x'',x_n),\ \forall\ x\in \br^n_+,
	\end{equation*}
	moreover, 
	\begin{equation*}
	\frac{\partial u^i}{\partial x_1}<0,\ \forall x\in\{x\in \br^n_+\ |\ x_1<t_1\}.
	\end{equation*}
\end{thm}

\vskip 0.2in

Especially, if we substitute \ref{Fanti} to the symmetric one \ref{Fsymall}, then by using \cref{thm:half} respect to all directions in $\br^n$, we have:
\begin{coro}
	Assume ${\bf f}$ satisfy \ref{Fc}, \ref{Fp}, \ref{Ffc}, \ref{Fsymall}, \ref{FH_i} and \ref{Fziinfh},  ${\bf h}$ satisfy \ref{H}. Let ${\bf u}=(u^1,\cdots,u^m)$ be a group of $[C^2(\overline{\br^n_+},\br_+)]^m$ strictly convex solutions to \eqref{eq:halfN} satisfying \eqref{eq:inf1} and \eqref{eq:inf2}, then each $u^i$ must be radially symmetric and strictly increasing respect to some point in $\br^n_+$. 
	
	More precisely, denote the rotating center as $x^*=((x^*)',x^*_n)\in\br^n_+$, and $r=|x'-(x^*)'|$, then for $i=1\dots,m$, each $u^i$ must be like
	\begin{equation*}
		u^i(x)=u^i(r,x_n),\ \forall\ x\in \br^n_+,
	\end{equation*}
	moreover, 
	\begin{equation*}
		\frac{\partial u^i}{\partial r}(r,x_n)>0,\ \forall\ x\in \br^n_+.
	\end{equation*}
\end{coro}

\vskip 0.2in
\subsection{Proof of \cref{thm:half}}

In what follows, we shall use the method of moving plane. We start by considering hyper-planes parallel to $\{x_1=0\}$, coming from $-\infty$. To proceed, for each $\lambda\leq 0$, we define 
\begin{align*}
\Sigma_\lambda&:=\{x\in\br^n_+|x_1<\lambda\},\\
T_\lambda&:=\{x\in\br^n_+|x_1=\lambda\},\\
P_\lambda&:=\partial \Sigma_\lambda\setminus T_\lambda=\{x=(x_1,\cdots,x_{n-1},0)\in\br^n_+|x_1<\lambda\}.
\end{align*}
For any point $x=(x_1,\cdots,x_{n-1},x_n)\in\Sigma_\lambda$, let $x_\lambda=(2\lambda-x_1,\cdots,x_{n-1},x_n)$ be the reflected point with respect to the plane $T_\lambda$. Define the reflected function by 
\begin{equation*}
u_\lambda(x):=u(x_\lambda)
\end{equation*}
and introduce the function 
\begin{equation*}
U_\lambda(x):=u_\lambda(x)-u(x) \text{ on }\overline{\Sigma_\lambda}.
\end{equation*}

By the same procedure in calculating \eqref{eq:EIinf}, we can obtain an elliptic inequality of $U^i_\lambda$ in $\Sigma_\lambda$:
$\forall x\in\Sigma_\lambda$ such that $\frac{\partial u^i}{\partial x_1}(x)\leq0$, we have
\begin{equation}\label{eq:EIH}
\begin{split}
\operatorname{tr}\left({\bf A^i}(x)D^2U^i_\Lambda(x)\right)+&{\bf B^i}(x)\cdot\nabla U^i_\lambda(x)\\
\geq&\sum_{j=1}^m \widetilde{d_{ij}}(x,u^1_\lambda,\cdots,u^j,\cdots,u^m,\nabla U^i_\lambda,U^j_\lambda)U^j_\lambda,
\end{split}
\end{equation}
where ${\bf A^i}(x):=(a^i_{jk}(x))_{j,k=1}^n$ with
\begin{equation*}
a^i_{jk}(x):=\int_0^1\det\left((1-t)D^2u^i_\lambda(x)+tD^2u^i(x)\right)\left((1-t)D^2u^i_\lambda(x)+tD^2u^i(x)\right)^{jk}dt,
\end{equation*}
are strictly positive definite due to the strictly convexity of $u^i$ and \ref{F0}, and together with $${\bf B^i(x)}:=\frac{h_{f^i,p}}{|\nabla U^i_\lambda(x))|}\chi_{\{|\nabla U^i_\lambda(x))|\neq0\}}\nabla U^i_\lambda(x)$$ are all locally bounded due to the twice differentiable continuity of $u^i$.

For all $\lambda<0$, the fact that $|x_\lambda|\leq|x|$ together with assumption \ref{H} lead us to $U^i_\lambda$ on the boundary that
\begin{equation}\label{eqT}
U^i_\lambda(x)=0, \hspace{0.5cm} x\in\overline{T_\lambda},
\end{equation}
\begin{equation}\label{eqP}
\frac{\partial U^i_\lambda}{\partial x_n}=h^i(|x_\lambda|)-h^i(|x|)\geq 0,  \hspace{0.5cm}x\in P_\lambda.
\end{equation}

Next we prove a preliminary lemma.

\begin{lemma}\label{lemma:half2}
	Let ${\bf u}=(u^1,\cdots,u^m)\in\left[C^2(\overline{\br^n_+},\br_+)\right]^m$ be a group of strictly convex solutions to \eqref{eq:halfN} satisfying the conditions in \cref{thm:half}, then $\forall\lambda<0$, $i=1,\dots,m$, if the maximum point of $U^i_\lambda$ on $\overline{\Sigma_\lambda}$ exists and the maximum is positive, then such point must be lying inside $\Sigma_\lambda$. That is, if $U^i_\lambda(y_1)=\max\limits_{y\in\overline{\Sigma_\lambda}}U^i_\lambda(y)>0$, then $y_1\in\Sigma_\lambda$.
\end{lemma}
\begin{proof}

Let $y_1\in\partial\Sigma_\lambda$ such that  $U^i_\lambda(y_1)=\max\limits_{y\in\overline{\Sigma_\lambda}}U^i_\lambda(y)>0$. By \eqref{eqT}, we only need to show that $y_1\notin P_\lambda$. 
	
By continuity, we can choose $r>0$ such that 
	\begin{equation*}
		U^i_\lambda(y)>\frac{1}{2}U^i_\lambda(y_1)>0,\hspace{0.1cm}\forall y\in B_r(y_1)\cap\overline{\br_+^n} .
	\end{equation*}
Furthermore, since $\left.U^i_\lambda\right|_{T_\lambda}=0$, we can suppose there exists at least one another point $y_2\in \overline{B_r(y_1)\cap \br_+^n}$ such that $U^i_\lambda(y_2)<\frac{3}{4}U^i_\lambda(y_1)$, this ensure that $U^i_\lambda(x)$ will not be constant in $B_r(y_1)\cap \br_+^n$.
	
	Let $\overline{U^i_\lambda}(x)=U^i_\lambda(x)-C^i$, where $C^i>0$ to be determined. By \eqref{eq:EIH}, we deduce the elliptic inequality of $\overline{U^i_\lambda}$ on $B_r(y_1)\cap \br_+^n$,
	\begin{equation}\label{eq:EIHC}
	\begin{split}
	&\operatorname{tr}\left({\bf A^i}(x)D^2\overline{U^i_\lambda}(x)\right)+{\bf B^i}(x)\cdot\nabla \overline{U^i_\lambda}(x)-\widetilde{d_{ii}}(x)\overline{U^i_\lambda}(x)\\
	\geq&\sum_{j\neq i} \widetilde{d_{ij}}(x,u^1_\lambda,\cdots,u^j,\cdots,u^m,\nabla U^i_\lambda,U^j_\lambda){U^j_\lambda}(x)+\widetilde{d_{ii}}(x)C^i,
	\end{split}
	\end{equation}
	
	For $j=1,\dots,m$, each $u^j(x)\in C^1(\overline{B_r(y_1)})$, therefore $u^j_\lambda$, $U^j_\lambda$, and $\nabla U^i_\lambda$ are bounded on $\overline{B_r(y_1)}$, by continuity of $f^i$, so are they on $\overline{B_r(y_1)}$. Hence exists $M^i_j>0$, such that for all $x\in\overline{B_r(y_1)}$,  
	\begin{equation*}
		\begin{split}
			&\left|\widetilde{d_{ij}}(x,u^1_\lambda,\cdots,u^j,\cdots,u^m,\nabla U^i_\lambda,U^j_\lambda){U^j_\lambda}\right|\\
			=&\left|f^i(x,u^1_\lambda,\cdots,u^j_\lambda,\cdots,u^m,\nabla U^i_\lambda)-f^i(x,u^1_\lambda,\cdots,u^j,\cdots,u^m,\nabla U^i_\lambda)\right|<2M^i_j.
		\end{split}
	\end{equation*}
	Hence by \ref{Fziinfh}, $\min\limits_{x\in\overline{B_r(y_1)}}\widetilde{d_{ii}}(x)>0$, choosing $C^i>\max\left\{\frac{2\sum\limits_{j=1}^mM^i_j}{\min\limits_{x\in\overline{B_r(y_1)}}\widetilde{d_{ii}}(x)}, U^i_\lambda(y_1)\right\}>0$, such that
	\begin{align*}
		\sum_{j\neq i} \widetilde{d_{ij}}(x,u^1_\lambda,\cdots,u^j,\cdots,u^m,\nabla U^i_\lambda,U^j_\lambda)&{U^j_\lambda}(x)+\widetilde{d_{ii}}(x)C^i\\
		\geq&-2\sum_{j=1}^mM^i_j+\min\limits_{x\in\overline{B_r(y_1)}}\widetilde{d_{ii}}(x)C^i>0.
	\end{align*}
	Hence, \eqref{eq:EIHC} becomes
	\begin{equation*}
		\operatorname{tr}\left({\bf A^i}(x)D^2\overline{U^i_\lambda}(x)\right)+{\bf B^i}(x)\cdot\nabla \overline{U^i_\lambda}(x)-\widetilde{d_{ii}}(x)\overline{U^i_\lambda}(x)>0.
	\end{equation*}
	By using \cref{lem:SMP-HL} on $\overline{U^i_\lambda}(x)<0$ over $B_r(y_1)\cap \br_+^n$, noting that $y_1$ is the maximum point of $\overline{U^i_\lambda}(x)$ on $\overline{B_r(y_1)\cap \br_+^n}$, we have
	\begin{equation*}
		\frac{\partial \overline{U^i_\lambda}}{\partial x_n}(y_1)<0.
	\end{equation*}
	However, if $y_1\in P_\lambda$, combining with \eqref{eqP} and by directly calculating we have
	\begin{equation*}
		\frac{\partial \overline{U^i_\lambda}}{\partial x_n}(y_1)=\frac{\partial U^i_\lambda}{\partial x_n}(y_1)\geq0,
	\end{equation*}
	which is a contradiction.
\end{proof}

\vskip 0.2in

We are now in a position to prove \cref{thm:half}. 

\newenvironment{prf5}{{\noindent\bf Proof of \cref{thm:half}.}}{\hfill $\square$\par}
\begin{prf5}
We will apply the moving plane method in three steps.\\

{\bf Step 1}: there exist a real number $\lambda\in\br$ such that $\left.U^i_\mu\right|_{\Sigma_\mu}\leq0, \forall\ i=1,\dots,m, \forall\ \mu<\lambda$.

Assume for contradiction that $\forall \lambda<0$, there exists $y_0\in\Sigma_\lambda$ such that $U^{i_0}_\lambda(y_0)>0$ for some $i_0$. We set $J:=\{j=1,\dots,m\ |\ U^j_\lambda\leq0,\ \forall x\in\Sigma_\lambda\}\subsetneqq\{1,\dots,m\}$, and $I:=\{1,\dots,m\}\setminus J.$ Up to a permutation of the indices, we can assume $I=\{1,\dots,l\},\ J=\{l+1,\dots,m\},\ l=|I|$. 

Note that $\sup\limits_{\Sigma_\lambda}U^i_\lambda>0$, $\forall\ i\in I$. By \cref{lemma:inf1}, we may take $\{y_i\}_{i\in I}\subset \overline{\Sigma_\lambda}\setminus T_\lambda$ such that 
\begin{equation*}
U^i_\lambda(y_i)=\max_{y\in\overline{\Sigma_\lambda}}U^i_\lambda(y)>0.
\end{equation*}

For fixed $i\in I$, by \cref{lemma:half2}, $y_i$ lies in the interior of $\Sigma_\lambda$, hence $\nabla U^i_\lambda(y_i)=0$, and $D^2 U^i_\lambda(y_i)\leq 0$. Then the $i$-th equation in \eqref{eq:EIH} at $y_i$ becomes
	\begin{align*}
		0\geq\sum_{j=1}^m \widetilde{d_{ij}}(y_i,u^1_\lambda(y_i),\cdots,u^j(y_i),\cdots,u^m(y_i),0,U^j_\lambda(y_i))U^j_\lambda(y_i),
		\stepcounter{equation}\tag{\theequation}\label{eq20}
	\end{align*}
	By \cref{rmk:dijinf}, definition of $J$ and $U^j_\lambda(y_i)\leq U^j_\lambda(y_j), \forall j\in I$, \eqref{eq20} becomes
	\begin{align*}
		0\geq\sum_{j\in I} \widetilde{d_{ij}}(y_i,u^1_\lambda,\cdots,u^j,\cdots,u^m,0,U^j_\lambda(y_i))U^j_\lambda(y_j).
		\stepcounter{equation}\tag{\theequation}\label{eq21}
	\end{align*}
	
	Next we will prove that $U^j_\lambda(y_j)\leq0$, $\forall j\in I$, which will be contradictory to the choice of $y_i$. Rewrite \eqref{eq21} as
	\begin{align*}
		MU=V,
		\stepcounter{equation}\tag{\theequation}\label{eq22}
	\end{align*}
	where $U:=\left(U^1_\lambda(y_1),\dots,U^l_\lambda(y_l)\right)$, $V:=(v_1,\dots,v_l)$, $M:=(m_{ij})_{i,j=1}^l$, satisfying $$v_i\leq0, m_{ij}:=\widetilde{d_{ij}}(y_i,u^1_\lambda,\cdots,u^j,\cdots,u^m,0,U^j_\lambda(y_i))\leq0,\ i,j=1,\dots,l.$$ 
	We can assume $\lambda<-\widetilde{R}$, where $\widetilde{R}$ as in \cref{lemma:inf2}, then $\Sigma_\lambda\subset\br^n\setminus B_{\widetilde{R}}(0)$, hence by \cref{lemma:inf2} and \cref{lemmaA}, $M$ is invertable, we can then solve $U=M^{-1}V$ by \eqref{eq22}, together with Cramer's law and \cref{lemmaA}, \cref{lemma:inf2}, we have
	$$U^j_\lambda(y_j)=M^{jk}v_k=\frac1{\det M}\sum_{k=1}^l\operatorname{adj}(M)_{jk}v_k\leq0,\ \forall j=1,\dots,l.$$\\
	
	Next we continue to move $T_\lambda$ by increasing $\lambda$, as long as $U^i_\lambda(x)\leq 0$ holds for all $i$ on $\Sigma_{\Lambda}$, we define
	\begin{equation*}
		\Lambda:=\sup\mleft\{\lambda\in\br\ \middle|\ U^i_\mu(x)\leq 0,\ \forall x\in\Sigma_\mu,\ \forall\ i=1,\dots,m,\ \forall\ \mu<\lambda\mright\}.
	\end{equation*}
	Step 1 shows that $\Lambda>-\infty$. On the other hand $\lim\limits_{|x|\to +\infty}u^i(x)=\infty$, for sufficiently large $R$, we have $u^i(x)>u^i(0)$, $\forall|x|>R$, hence $\Lambda<+\infty$, therefore $\Lambda$ is finite. 
	Since all the terms are continuous with respect to $\lambda$, we have 
	\begin{equation}\label{eq:halfstep1.5}
	U^i_\Lambda(x)\leq 0,\ \forall x\in\Sigma_\Lambda,\ \forall\ i=1,\dots,m.
	\end{equation}
	In particular, $U^i_\mu\leq 0, \forall \mu\leq\Lambda$, hence $\forall x\in\Sigma_\Lambda$, $\frac{\partial u^i}{\partial x_1}(x)\leq0$, by covering argument.
 By \eqref{eq:EIH} and \cref{rmk:dijinf}, we have on $\Sigma_\Lambda$,
	\begin{equation}\label{eq:EIHF}
	\operatorname{tr}\left({\bf A^i}(x)D^2U^i_\Lambda(x)\right)+{\bf B^i}(x)\cdot\nabla U^i_\Lambda(x)-\widetilde{d^{ii}}(x)U^i_\Lambda(x)\geq0,\ \forall\ i=1,\dots,m.
	\end{equation}
	
	Since all the coefficient of \eqref{eq:EIHF} are locally bounded, by \cref{lem:SMP-HL}, we have $\forall i=1,\dots,m$, either
	\begin{align*}
		&U^i_\Lambda(x)<0, \hspace{0.1cm}\forall x\in \Sigma_\Lambda,\\
		\text{or}\hspace{0.1cm}&U^i_\Lambda(x)\equiv 0, \hspace{0.1cm}\forall x\in \Sigma_\Lambda.
	\end{align*}
	If the former one happens, we also have
	\begin{align*}
		\frac{\partial U^i_\Lambda}{\partial x_1}(x)>0, \hspace{0.1cm}\forall x\in T_\Lambda.
	\end{align*}
	
	Noting that the Neumann boundary condition shows that $\left.U^i_\Lambda\right|_{P_\Lambda}<0$. In fact, assume $U^i_\Lambda|_{P_\Lambda}=0$, by\cref{lem:SMP-HL} we have
	\begin{align*}
		\frac{\partial U^i_\Lambda}{\partial x_n}(x)<0, \hspace{0.1cm}x\in P_\Lambda,
	\end{align*}
	which is contradictory to \eqref{eqP}. Hence 
	\begin{equation*}
		U^i_\Lambda(x)<0, \hspace{0.1cm}x\in P_\Lambda.
	\end{equation*}
	
	Now, if $\Lambda=0$, we have $\frac{\partial u^i}{\partial x_1}(0,x')=0$ for some $x\in \{x_1=0\}$, hence $U^i_\Lambda(x)\equiv 0, \hspace{0.1cm}\forall x\in \Sigma_\Lambda$.
	
	Next, we discuss the case of $\Lambda<0$.
	
	\textbf{Step 2}: We will prove that $U^i_\Lambda\equiv 0$ on $\Sigma_\Lambda$ for at least one $i$.
	
	As the discussion above, we only need to exclude the case that all the $U^i_\Lambda$ on $\Sigma_\Lambda$ are strictly negative, with directional derivative along $x_1$ being strictly positive on $T_\Lambda$. Assume this is happening.
	
	By definition of $\Lambda$, there exists a sequence of $\{\lambda_k\}_{k=1}^\infty\subset\br$ and a sequence of $\{y_k\}_{k=1}^\infty\subset\br^n_+$, such that $\lambda_k\in(\Lambda, 0)$, satisfying $\lim\limits_{k\to\infty}\lambda_k=\Lambda$, and $y_k\in\Sigma_{\lambda_k}$, such that at least for one $i$, $U^i_{\lambda_k}(y_k)\geq0$. By \cref{lemma:inf1}, we can choose $\{y_k\}\subset\overline{\Sigma_{\lambda_k}}\setminus\Sigma_{\Lambda}$, such that
	\begin{equation}\label{step2 max}
	U^i_{\lambda_k}(y_k)=\max_{y\in\overline{\Sigma_{\lambda_k}}}U^i_{\lambda_k}(y)>0, \hspace{0.1cm}k=1,2,\cdots
	\end{equation}
	There two cases will happen. 

\begin{enumerate}
	\renewcommand{\theenumi}{\textbf{Case \arabic{enumi}}}
	\renewcommand{\labelenumi}{\theenumi}
	\item  the sequence $\{y_k\}_{k=1}^\infty$ has bounded subsequence. 
	
	Without lost of generality, we can assume
	\begin{equation*}
		\lim_{k\to\infty}y_k=y_0\in\bigcap_{k=1}^{+\infty}\overline{\Sigma_{\lambda_k}}\setminus\Sigma_{\Lambda}=\overline{T_\Lambda},
	\end{equation*}
	
	We claim that $\frac{\partial U^i_\Lambda}{\partial x_1}(y_0)>0$.
	
	If $y_0\in T_\Lambda$, then by \cref{lem:SMP-HL} we are done. Next, consider $y_0\in\overline{P_\Lambda}\cap\overline{T_\Lambda}$. In fact, for any $Q\in\overline{P_\Lambda}\cap\overline{T_\Lambda}$, we have $U^i_\Lambda(Q)=0$, hence
	\begin{equation*}
		\frac{\partial U^i_\Lambda}{\partial x_1}(Q)\geq 0.
	\end{equation*}
	Let $s=-{\bf e_1}+{\bf e_n}$, at $Q$, locally we can regarded as $\mleft\{\rho\equiv x_1-\Lambda=0 \mright\}$ intersect with $\mleft\{\sigma\equiv x_n=0\mright\}$, noticing that \eqref{eq:det} and \eqref{eq:aij} shows that $a_{1j}=0,\forall j=2,\dots,n$, and then $a_{ij}\rho_i\sigma_j=0$, hence at $Q$, by using \cref{lemmaS} on \eqref{eq:EIHF}, we have either
	\begin{equation}\label{eq:H1}
	\frac{\partial U^i_\Lambda}{\partial s}(Q)=-\frac{\partial U^i_\Lambda}{\partial x_1}(Q)+\frac{\partial U^i_\Lambda}{\partial x_n}(Q)<0,
	\end{equation}
	or
	\begin{equation}\label{eq:H2}
	\frac{\partial^2 U^i_\Lambda}{\partial s^2}(Q)=\frac{\partial^2 U^i_\Lambda}{\partial x_1^2}(Q)-2\frac{\partial^2 U^i_\Lambda}{\partial x_1\partial x_n}(Q)+\frac{\partial^2 U^i_\Lambda}{\partial x_n^2}(Q)<0. 
	\end{equation}
	
	For \eqref{eq:halfN}, with \ref{H}, by direct calculations showing that
	\begin{align*}
		\frac{\partial^2 U^i_\Lambda}{\partial x_1\partial x_n}(Q)=-2\frac{\Lambda}{|Q|}\frac{\partial h^i}{\partial r}(|Q|)\leq 0, \stepcounter{equation}\tag{\theequation}\label{eq:H3}
	\end{align*}
	Hence by \eqref{nablaonT}, \eqref{HessainonT}, \eqref{eq:H2}, \eqref{eq:H3} , only \eqref{eq:H1} can hold, hence
	\begin{equation}\label{eqH6}
	\frac{\partial U^i_\Lambda}{\partial x_1}(Q)>0.
	\end{equation}
	Let $Q$ be $y_0$, we prove the claim. 
	
	Next, by mean value theorem $$0< U^i_{\lambda_k}(y_k)=u^i(y_k^{\lambda_k})-u^i(y_k)=2\frac{\partial u^i}{\partial x_1}(\xi_k)(\lambda_k-y_{k,1})<0,$$ where $\xi_k$ lies in the segment connecting $y_k$ and $y_k^{\lambda_k}$, noticing that the last strict less than sign can be achieved by sufficiently large $k$ pushing $\xi_k$ near $y_0$ close enough, which is a contradiction.
	
	\item  $\lim_{k\to\infty}|y_k|=+\infty$.  
        In this case, we can choose $k^*>0$ such that $\forall k>k^*$, $|y_k|>\widetilde{R}$, where $\widetilde{R}$ as in \cref{lemma:inf2}. Then for each $y_k, k>k^*$, the argument similar to \textbf{Step 1} will give a contradiction.
 \end{enumerate}
	
	\textbf{Step 3}: We need to prove all the $U^i_\Lambda\equiv 0$ on $\Sigma_\Lambda$.
	
	By \textbf{Step 2},	denote $I:=\mleft\{i\ \middle|\ \left.U_\Lambda^i\right|_{\Sigma_\Lambda}\equiv0\mright\}$, $J:=\{1,\dots,m\}\setminus I$. If $J\neq\emptyset$, by \ref{Ffc}, there exists $i_0\in I, j_0\in J,$ such that
	$$\left|\mleft\{x\in\br^n_+\ \middle|\ f^{i_0} \text{strictly decreasing with respect to} z^{j_0}, \text{as} z^k, k\neq i_0,j_0, p\text{fixed}\mright\}\right|>0,$$ that is there exists$x_0\in\br^n_+$, such that $$\widetilde{d_{i_0j_0}}(x_0,u^1_\Lambda(x_0),\cdots,u^j(x_0),\cdots,u^m(x_0),0,U^{j_0}_\Lambda(x_0))U^{j_0}_\Lambda(x_0)<0.$$
	
	If $x_0\in\Sigma_\Lambda$, then by \cref{lemma:SMPFinf} and \eqref{eq:EIH}, together with $\widetilde{d_{ij}}\leq 0,i\neq j$ in \cref{rmk:dijinf}, and \eqref{eq:halfstep1.5} show that
	\begin{align*}
	0\geq&\sum\limits_{j\neq i_0} \widetilde{d_{i_0j}}U^j_\Lambda=\sum\limits_{j\in J} \widetilde{d_{i_0j}}U^j_\Lambda\\
	\geq&\widetilde{d_{i_0j_0}}U^{j_0}_\Lambda>0,
	\end{align*}
	which is a contradiction.
	
	If $x_0\in\Sigma_\Lambda^\Lambda$, consider the elliptic inequality of $\Sigma_\Lambda^\Lambda$ on $U^i_\Lambda$ should be
	\begin{equation*}
	\begin{split}
	\operatorname{tr}\left({\bf A^i}(x)D^2U^i_\Lambda(x)\right)+&{\bf B^i}(x)\cdot\nabla U^i_\Lambda(x)\\
	\geq&\sum_{j=1}^m -\widetilde{d_{ij}}(x,u^1_\Lambda,\cdots,u^j,\cdots,u^m,\nabla U^i_\Lambda,U^j_\Lambda)U^j_\Lambda.
	\end{split}
	\end{equation*}
	Noting that for $j\in J$, $U^j_\Lambda$ should be strictly positive on $\Sigma_\Lambda^\Lambda$, hence
    \begin{align*}
	0\geq&\sum\limits_{j\neq i_0} -\widetilde{d_{i_0j}}U^j_\Lambda=\sum\limits_{j\in J} -\widetilde{d_{i_0j}}U^j_\Lambda\\
	\geq& -\widetilde{d_{i_0j_0}}U^{j_0}_\Lambda>0,
	\end{align*}
	which figures out a contradiction again. Hence $J=\emptyset$, that is all the $U^i_\Lambda\equiv 0$ hold on $\Sigma_\Lambda$.
\end{prf5}

\section{Unbounded Tube Shape Domains}\label{sec:ubt}
\noindent

Now we turn our attention on the unbounded tubes, any tubes in $\br^n$ can always regarded as $C_{\infty}:=\Omega\times(-\infty,\infty)$ up to rotations and translations, where $\Omega\subset\br^{n-1}$ is a bounded simply connected domain, we assume $\partial\Omega\in C^2$. 

In this case, we mainly consider the following constant-boundary Dirichlet problem for \eqref{eq:MA} on $C_{\infty}$,
\begin{equation}\label{eq:cylinderinf}
\left\{\begin{array}{rl}
\det(D^2 u^i)\ =&f^i(x,{\bf u},\nabla u^i), x\in C_{\infty},\\
u^i\ =&h^i,\hspace{1.8cm}x\in \partial C_{\infty},\\
\lim\limits_{|x|\to +\infty}u^i(x)\ =&\infty, \hspace{1.8cm} i=1,\dots,m.\\
\end{array}
\right.
\end{equation}
where $h^i$ satisfy $\lim\limits_{|x_n|\to +\infty}h^i(x)=\infty.$

The behavior of $u^i$ at the infinity as above are natural since we always assume $u^i$ to be strictly convex.

In this case, we need only a weaken condition instead of \eqref{eq:inf1}:

For $u^i\in C^1(\br^n)$, we denote
 $$\frac{\partial u^i}{\partial |x_n|}(x):=\langle\nabla u^i(x), {\bf e_n}\rangle\langle{\bf e_n}, \frac{x}{|x|}\rangle,$$ $$\nabla_{\tau'}(x):=\nabla u^i(x)-\frac{\partial u^i}{\partial |x_n|}(x)\frac{x}{|x|},$$ We require that
\begin{align*}
	&\lim\limits_{|x_n|\to\infty}\frac{\partial u^i}{\partial |x_n|}(x)>0,\\
	&|\nabla_{\tau'}(x)|=o\left(\frac{\partial u^i}{\partial |x_n|}(x)\right)\ \text{as}\ |x_n|\to\infty.\stepcounter{equation}\tag{\theequation}\label{eq:inf3}
\end{align*}

\vskip 0.2in

Similar to the case of bounded tubes \cite{zhang2024}, when we move the hyperplane $\mleft\{x_1=\lambda\mright\}$ along ${\bf e_1} $  from left negative infinity, denote $\lambda_0:=\inf\mleft\{x_1\ \middle|\ x\in C_\infty \mright\}$. Then we can define similarly the $$\Lambda_0:=\sup\mleft\{\lambda>\lambda_0\ \middle|\  \Sigma_\lambda^\lambda\subset\Omega\mright\},$$ with two probably happened situations, and define $\Lambda_1,\Lambda_2$ respectively.

\vskip 0.2in
\subsection{Main Theorem}

We now state our main theorem for case of unbounded tubes.

When $\Omega$ being convex in one direction, assumed as $\bf e_1$, and the nonlinear term $\bf f$ satisfies some corresponding monotonic conditions in this direction, we can start to examine whether the solution of the system \eqref {eq:cylinderinf} will satisfy the corresponding monotonicity along this direction. The main results are the following.

\begin{thm}\label{thm:cylinderinf}
	Let $C_{\infty}=\Omega\times(-\infty, \infty)$ be a unbounded tubes in $\br^n$, where $\Omega\subset\br^{n-1}$ is a $C^2$ bounded simply connected domain being convex along with ${\bf e_1}$. Assume ${\bf f}$ satisfy \ref{Fc}, \ref{Fziinf}, \ref{Fzj}, \ref{Fp}, \ref{Fanti} and \ref{FH_i}. Let ${\bf u}=\left(u^1,\cdots,u^m\right)\in \left[C^2\left(\overline{C_{\infty}}\right)\right]^m$ be a group of strictly convex solutions to \eqref{eq:cylinderinf} satisfying \eqref{eq:inf2} and \eqref{eq:inf3}, then each $u^i$ on $\mleft\{x\in C_{\infty}\ \middle|\ x_1<\Lambda_0\mright\}$ must be
	$$u^i\left(x_1,x'\right)\geq u^i\left(2\Lambda_0-x_1,x'\right)\ \text{and}\ \frac{\partial u^i}{\partial x_1}(x)<0.$$
	
	Furthermore, if
	\begin{equation}\label{eq:nd=0cylinderinf}
	\frac{\partial u^i}{\partial x_1}\left(\Lambda_0,x'\right)=0 ~for~some~ x\in \left\{x_1=\Lambda_0\right\},
	\end{equation} then such $u^i$ must be symmetric with respect to $\left\{x_1=\Lambda_0\right\}$ and strictly decreasing in ${\bf e_1}$ direction with $x_1<\Lambda_0$, more precisely, 
	\begin{equation*}
	u^i(x)=u^i\left(\left|x_1-\Lambda_0\right|,x'\right),\ \text{in}\ \mleft\{x\in C_{\infty}\ \middle|\ \left|x_1-\Lambda_0\right|<\Lambda_0-\lambda_0\mright\},
	\end{equation*}
	moreover, 
	\begin{equation*}
	\frac{\partial u^i}{\partial x_1}(x)<0,\ \text{in}\ \mleft\{x\in C_{\infty}\ \middle|\ x_1<\Lambda_0\mright\}.
	\end{equation*}
\end{thm}

\vskip 0.2in

If we assume more symmetry condition on $\Omega$ and ${\bf f}$ (substituting \ref{Fanti} with \ref{Fsym1}) to satisfy \eqref{eq:nd=0cylinderinf}, we can furthermore immediately have the following, by using \cref{thm:cylinderinf} again with ${\bf u_{\Lambda_0}}:=(u^i_{\Lambda_0})$. (Noting that in this case, the inequalities \eqref{eq:anti},\eqref{eq:splitinf} will be slightly different to obtain the same result.)

\begin{thm}
	Let $C_{\infty}=\Omega\times(-\infty, \infty)$ be a unbounded tubes in $\br^n$, where $\Omega\subset\br^{n-1}$ is a $C^2$ bounded simply connected domain being convex along with ${\bf e_1}$. Assume ${\bf f}$ satisfy \ref{Fc}, \ref{Fziinf}, \ref{Fzj}, \ref{Fp}, \ref{Fsym1} and \ref{FH_i}. Let ${\bf u}=\left(u^1,\cdots,u^m\right)\in \left[C^2\left(\overline{C_{\infty}}\right)\right]^m$ be a group of strictly convex solutions to \eqref{eq:cylinderinf} satisfying \eqref{eq:inf2} and \eqref{eq:inf3}, then each $u^i$ must be symmetric with respect to $\left\{x_1=\Lambda_0\right\}$, and strictly decreasing in $\left\{x_1<\Lambda_0\right\}$ in ${\bf e_1}$ direction with $x_1<\Lambda_0$.
	
	More precisely, for all $i=1,\dots, m$, 
	\begin{equation*}
	u^i(x)=u^i\left(\left|x_1-\Lambda_0\right|,x'\right),\ \text{in}\ \mleft\{x\in C_{\infty}\ \middle|\ \left|x_1-\Lambda_0\right|<\Lambda_0-\lambda_0\mright\},
	\end{equation*}
	moreover, 
	\begin{equation*}
	\frac{\partial u^i}{\partial x_1}(x)<0,\ \text{in}\ \mleft\{x\in C_{\infty}\ \middle|\ x_1<\Lambda_0\mright\}.
	\end{equation*}
\end{thm}

\vskip 0.2in

Especially, when $\Omega$ is a ball, that is $C_\infty$ being a infinite cylinder, if we substitute \ref{Fanti} to the symmetric ones \ref{Fsymall}, then by using \cref{thm:cylinderinf} respect to all directions in $\br^n$, we have immediately that

\begin{coro}\label{thm:cylinderinfball}
	Let $C_{\infty}=\Omega\times(-\infty, \infty)$ be a infinite cylinder in $\br^n$, $\Omega=B_R$ be a arbitrary ball with radius $R$ in $\br^{n-1}$. Assume ${\bf f}$ satisfy \ref{F0}, \ref{Fziinf}, \ref{Fzj}, \ref{Fp}, \ref{Fsymall} and \ref{FH_i}. Let ${\bf u}=\left(u^1,\cdots,u^m\right)\in \left[C^2\left(\overline{C_{\infty}}\right)\right]^m$ be a group of strictly convex solutions to \eqref{eq:cylinderinf} satisfying \eqref{eq:inf2} and \eqref{eq:inf3}, then each must be radially symmetric and strictly increasing respect to the axis crossing the center of $B_R$.
	
	More precisely, denote the center of $B_R$ as $x^*=\left((x^*)',x^*_n\right)\in\br^n$, and denote $r=|x'-(x^*)'|$, then for $i=1\dots,m$, each $u^i$ must be
	\begin{equation*}
	u^i(x)=u^i(r,x_n),\ x\in C_{\infty},
	\end{equation*}
	moreover, 
	\begin{equation*}
	\frac{\partial u^i}{\partial r}(r,x_n)>0,\ x\in C_{\infty}.
	\end{equation*}
\end{coro}

\vskip 0.2in
\subsection{Proof of \cref{thm:cylinderinf}}

By the same procedure in calculating \eqref{eq:EIinf}, we can obtain an elliptic inequality of $U^i_\lambda$ in $\Sigma_\lambda$:
$\forall x\in\Sigma_\lambda$ such that $\frac{\partial u^i}{\partial x_1}(x)\leq0$, we have
\begin{equation}\label{eq:EIC}
\begin{split}
\operatorname{tr}\left({\bf A^i}(x)D^2U^i_\Lambda(x)\right)+&{\bf B^i}(x)\cdot\nabla U^i_\lambda(x)\\
\geq&\sum_{j=1}^m \widetilde{d_{ij}}(x,u^1_\lambda,\cdots,u^j,\cdots,u^m,\nabla U^i_\lambda,U^j_\lambda)U^j_\lambda,
\end{split}
\end{equation}
where ${\bf A^i}(x):=\left(a^i_{jk}(x)\right)_{j,k=1}^n$,
\begin{equation*}
a^i_{jk}(x):=\int_0^1\det\left((1-t)D^2u^i_\lambda(x)+tD^2u^i(x)\right)\left((1-t)D^2u^i_\lambda(x)+tD^2u^i(x)\right)^{jk}dt,
\end{equation*}
are strictly positive definite due to the strictly convexity of $u^i$ and \ref{F0}, and together with $${\bf B^i(x)}:=\frac{h_{f^i,p}}{|\nabla U^i_\lambda(x))|}\chi_{\{|\nabla U^i_\lambda(x))|\neq0\}}\nabla U^i_\lambda(x)$$ are all locally bounded due to the twice differentiable continuity of $u^i$.

\vskip 0.2in

Next we prove a preliminary lemma.

\begin{lemma}\label{lemma:cylinderinf1}
	Let ${\bf u}=(u^1,\cdots,u^m)$ satisfy \eqref{eq:inf2} and \eqref{eq:inf3}, then for all $\lambda<0$, 
	$$\lim\limits_{|x_n|\to\infty,x_1<\lambda}U^i_\lambda(x)\leq0, \forall i=1,\dots,m.$$ 
\end{lemma}
\begin{proof}
	By \eqref{eq:inf3}
	\begin{align*}
		\frac{\partial u^i}{\partial x_1}&=\frac{\partial u^i}{\partial |x_n|}\frac{x_1}{|x|}+\left(\nabla u^i-\frac{\partial u^i}{\partial |x_n|}\frac{x}{|x|}\right)\cdot{\bf e_1}\\
		&=\frac{\partial u^i}{\partial |x_n|}\left(\frac{x_1}{|x|}+o(1)\right),\ \text{as}\ |x_n|\to\infty,\ \forall\ i=1,\dots,m,
	\end{align*}
	hence for all $\lambda<0$,
	\begin{align*} U^i_\lambda(x)&=u^i(2\lambda-x_1,x')-u^i(x)\\
		&=u^i(2\lambda-x_1,x')-u^i(-x_1,x')+u^i(-x_1,x')-u^i(x)\\
		&=2\lambda\int_0^1\frac{\partial u^i}{\partial x_1}\left(2(1-t)\lambda-x_1,x'\right)dt+u^i(-x_1,x')-u^i(x)\\
		&=2\lambda\int_0^1\frac{\partial u^i}{\partial |x_n|}\left(\frac{2(1-t)\lambda-x_1}{|x|}+o(1)\right)dt+u^i(-x_1,x')-u^i(x),\ \text{as}\ |x_n|\to\infty,
	\end{align*}
	Noticing that $\frac{\partial u^i}{\partial |x_n|}>0$, as $|x_n|\to\infty$, hence for $x_1$ sufficiently small such that $x_1<2\lambda$, by \eqref{eq:inf2}, we have$$\lim\limits_{|x_n|\to\infty,x_1<2\lambda}U^i_\lambda(x)\leq0.$$ 
	On the other hand, for $\lambda<0$,
	\begin{align*} U^i_\lambda(x)&=u^i(2\lambda-x_1,x')-u^i(x_1,x')\\
		&=2(\lambda-x_1)\int_0^1\frac{\partial u^i}{\partial x_1}\left(2t(\lambda-x_1)+x_1,x'\right)dt\\
		&=2(\lambda-x_1)\int_0^1\frac{\partial u^i}{\partial |x_n|}\left(\frac{2t(\lambda-x_1)+x_1}{|x|}+o(1)\right)dt,\ \text{as}\ |x_n|\to\infty,
	\end{align*}
	Noticing that $\frac{\partial u^i}{\partial |x_n|}>0$, as $|x_n|\to\infty$, hence for $x_1$ sufficiently small such that $2\lambda<x_1<\lambda$, we have
	$$\lim\limits_{|x_n|\to\infty,2\lambda<x_1<\lambda}U^i_\lambda(x)\leq0.$$ 
\end{proof}

\begin{rmk}
This lemma is in order to ensure the super-mum of $U_\lambda^i$ on $\Sigma_\lambda$ can be achieved inside $\Sigma_\lambda$, not just a sequence of maximizers along $x_n$ being infinity. This version is the axis symmetric case, in fact we can weaken the condition \eqref{eq:inf2} if we only consider one direction.
\end{rmk}

\vskip 0.2in

We are now in a position to prove \cref{thm:cylinderinf}. 

\newenvironment{prf6}{{\noindent\bf Proof of \cref{thm:cylinderinf}.}}{\hfill $\square$\par}
\begin{prf6}
We will apply the moving plane method in three steps.\\

{\bf Step 1}: there exist a real number $\lambda\in\br$ such that $\left.U^i_\mu\right|_{\Sigma_\mu}\leq0, \forall\ i=1,\dots,m, \forall\ \mu<\lambda$.
	
	Firstly, consider $\Sigma_\lambda^{\widetilde{R}}:=\Sigma_\lambda\cap\{|x|\leq \widetilde{R}\}$, where $\widetilde{R}$ is given by \cref{lemma:inf2}. Without lost of generality, we can assume $\partial\Sigma_\lambda^{\widetilde{R}}\cap C_{\infty}$ being $C^2$.(If not, we can generate a larger $C^2$ surface along with the boundary covering $\Sigma_\lambda^{\widetilde{R}}$.)

	For any $x\in\{x_1=\lambda_0\}\cap\partial C_{\infty}$, by maximum principle, locally we have $\frac{\partial u^i}{\partial x_1}(x)<0$, hence by the continuity of $\nabla u^i$ up to the boundary, for sufficiently small $\epsilon_x>0$, we must have $$\frac{\partial u^i}{\partial x_1}<0, x\in  C_{\infty}\cap B_{\epsilon_x}(x),\ \forall i=1,\dots,m.$$ 
	 Noting that $\{x_1=\lambda_0\}\cap\partial C_{\infty}$ is a compact set, hence we must have finite cover of it by $\{B_{\epsilon_x}(x)\}$, denoted as $\{B_{\epsilon_{x_i}}(x_i)\}_{i=1}^K$. 
	
	Now denote $$ C_{\infty}^\epsilon:= C_{\infty}\cap\bigcup\limits_{i=1}^K B_{\epsilon_{x_i}}(x_i),$$ we have $$\frac{\partial u^i}{\partial x_1}(x)<0,\ x\in  C_{\infty}^\epsilon,\ \forall\ i=1,\dots,m, $$ hence when $\lambda$ sufficiently close to $\lambda_0$ (such that $\Sigma_\lambda\cup\Sigma_\lambda^\lambda\subset  C_{\infty}^\epsilon$, fix such $\lambda$, denote as $\lambda'$, For any $\mu\in(\lambda_0,\lambda')$, we have $$U^i_\mu(x)=\int_{x_1}^{2\mu-x_1}\frac{\partial u^i}{\partial x_1}(s,x')ds<0,\ x\in \Sigma_\mu.$$
	
    Consider $\Sigma_\lambda^{\widetilde{R}^c}:=\Sigma_\lambda\setminus\{|x|\leq \widetilde{R}\}$, we assume that the goal of the step is not true, that is, assume $\forall \lambda<0$, exists $y_0\in\Sigma_\lambda^{\widetilde{R}^c}$ such that for some $i_0$, $U^{i_0}_\lambda(y_0)>0$. We denote the index set $J:=\{j=1,\dots,m\ |\ U^j_\lambda(x)\leq0,\ \forall x\in\Sigma_\lambda^{\widetilde{R}^c}\}\subsetneqq\{1,\dots,m\}$, and $I:=\{1,\dots,m\}\setminus J.$ Without lost of generality, we can assume $I=\{1,\dots,l\},\ J=\{l+1,\dots,m\},\ l=|I|$. 
	
	Noticing that $\sup\limits_{\Sigma_\lambda}U^i_\lambda>0$, $\forall\ i\in I$. By \cref{lemma:cylinderinf1}, we can choose $\{y_i\}_{i\in I}\subset\overline{\Sigma_\lambda^{\widetilde{R}^c}}$ such that
	\begin{equation*}
	U^i_\lambda(y_i)=\max_{y\in\overline{\Sigma_\lambda^{\widetilde{R}^c}}}U^i_\lambda(y)>0.
	\end{equation*}
	For fixed $i\in I$, by $\left.U^i_\lambda\right|_{T_{\lambda}}=0$, $y_i\in\Sigma_\Lambda^{\widetilde{R}^c}$, hence $\nabla U^i_\lambda(y_i)=0$, and $D^2 U^i_\lambda(y_i)\leq 0$. Therefore the $i$-th equation in \eqref{eq:EIC} at $y_i$ becomes
	\begin{align*}
	0\geq\sum_{j=1}^m \widetilde{d_{ij}}(y_i,u^1_\lambda(y_i),\cdots,u^j(y_i),\cdots,u^m(y_i),0,U^j_\lambda(y_i))U^j_\lambda(y_i),
	\stepcounter{equation}\tag{\theequation}\label{eq10}
	\end{align*}
	by \cref{rmk:dijinf}, definition of $J$ and $U^j_\lambda(y_i)\leq U^j_\lambda(y_j), \forall j\in I$, \eqref{eq10} becomes
	\begin{align*}
	0\geq\sum_{j\in I} \widetilde{d_{ij}}(y_i,u^1_\lambda,\cdots,u^j,\cdots,u^m,0,U^j_\lambda(y_i))U^j_\lambda(y_j).
	\stepcounter{equation}\tag{\theequation}\label{eq11}
	\end{align*}
	
	We will prove next that $U^j_\lambda(y_j)\leq0$, $\forall j\in I$, which will be a contradiction to the choice of $y_i$. Rewrite \eqref{eq11} as
	\begin{align*}
	MU=V,
	\stepcounter{equation}\tag{\theequation}\label{eq12}
	\end{align*}
	where $U:=\left(U^1_\lambda(y_1),\dots,U^l_\lambda(y_l)\right)$, $V:=(v_1,\dots,v_l)$, $M:=(m_{ij})_{i,j=1}^l$, satisfy $$v_i\leq0, m_{ij}:=\widetilde{d_{ij}}(y_i,u^1_\lambda,\cdots,u^j,\cdots,u^m,0,U^j_\lambda(y_i))\leq0,\ i,j=1,\dots,l.$$ 
	Noticing that $\Sigma_\Lambda^{\widetilde{R}^c}\subset\br^n\setminus B_{\widetilde{R}}(0)$, by \cref{lemma:inf2} and \cref{lemmaA}, $M$ is invertable, hence by \eqref{eq12}we can solve $U=M^{-1}V$, by Cramer's law and \cref{lemmaA}, together with \cref{lemma:inf2}, we have
	$$U^j_\lambda(y_j)=M^{jk}v_k=\frac1{\det M}\sum_{k=1}^l\operatorname{adj}(M)_{jk}v_k\leq0,\ \forall j=1,\dots,l,$$ which is contradictory to the choice of $y_i$. Hence exists $\lambda''<0$ such that $\left.U^i_\mu\right|_{\Sigma_\mu^{\widetilde{R}^c}}\leq0, \forall\ i=1,\dots,m, \forall\ \mu<\lambda,$ 
		
	Take $\lambda:=\min\{\lambda',\lambda''\}$ we finish \textbf{Step 1}.\\
	
	Next we continue to move $T_\lambda$ by increasing $\lambda$, as long as $U^i_\lambda(x)\leq 0$ holds for all $i$ on $\Sigma_{\Lambda}$. We define
	\begin{equation*}
	\Lambda:=\sup\mleft\{\lambda\in\br\ \middle|\ U^i_\mu(x)\leq 0,\ \forall x\in\Sigma_\mu,\ \forall\ i=1,\dots,m,\ \forall\ \mu<\lambda\mright\}.
	\end{equation*}
	Since all the terms are continuous with respect to $\lambda$, we have 
	\begin{equation}\label{eq:cylinderinfstep1.5}
	U^i_\Lambda(x)\leq 0,\ \forall x\in\Sigma_\Lambda,\ \forall\ i=1,\dots,m.
	\end{equation}
	In particular, $U^i_\mu\leq 0, \forall \mu\leq\Lambda$, hence by covering argument, $\forall x\in\Sigma_\Lambda$, we have $\frac{\partial u^i}{\partial x_1}(x)\leq0$,
	hence by \eqref{eq:EIC} and \cref{rmk:dijinf}, on $\Sigma_\Lambda$,
	\begin{equation}\label{eq:EICF}
	\operatorname{tr}\left({\bf A^i}(x)D^2U^i_\Lambda(x)\right)+{\bf B^i}(x)\cdot\nabla U^i_\Lambda(x)-\widetilde{d^{ii}}(x)U^i_\Lambda(x)\geq0,\ \forall\ i=1,\dots,m.
	\end{equation}
	
	Since all the coefficient of \eqref{eq:EICF} are locally bounded, by \cref{lem:SMP-HL}, we have $\forall i=1,\dots,m$, either
	\begin{align*}
	&U^i_\Lambda(x)<0, \hspace{0.1cm}\forall x\in \Sigma_\Lambda,\\
	\text{or}\hspace{0.1cm}&U^i_\Lambda(x)\equiv 0, \hspace{0.1cm}\forall x\in \Sigma_\Lambda.
	\end{align*}
	If the former one happens, we also have
	\begin{align*}
	\frac{\partial U^i_\Lambda}{\partial x_1}(x)>0, \hspace{0.1cm}\forall x\in T_\Lambda.
	\end{align*}
	
	\textbf{Step 2}: We need to prove $\Lambda=\Lambda_0$. 
	
	By definition of $\Lambda$, there exists a sequence of $\{\lambda_k\}_{k=1}^\infty\subset\br$ and a sequence of $\{y_k\}_{k=1}^\infty\subset\br^n$, such that $\lambda_k\in(\Lambda, 0)$, satisfying $\lim\limits_{k\to\infty}\lambda_k=\Lambda$, and $y_k\in\Sigma_{\lambda_k}$, such that at least for one $i$, $U^i_{\lambda_k}(y_k)>0$. By \cref{lemma:cylinderinf1}, we can choose $\{y_k\}\subset\overline{\Sigma_{\lambda_k}}\setminus\Sigma_{\Lambda}$, such that 
	\begin{equation}\label{cstep2 max}
	U^i_{\lambda_k}(y_k)=\max_{y\in\overline{\Sigma_{\lambda_k}}}U^i_{\lambda_k}(y)>0, \hspace{0.1cm}k=1,2,\cdots
	\end{equation}
There two cases will happen. 

\begin{enumerate}
	\renewcommand{\theenumi}{\textbf{Case \arabic{enumi}}}
	\renewcommand{\labelenumi}{\theenumi}
	\item the sequence $\{y_k\}_{k=1}^\infty$ exists bounded subsequence, that is $x_k\to x^*\in\overline{\Sigma_\Lambda}$.
	
	By continuity of $u^i$, $U^i_{\Lambda}(x^*)\geq 0$, Hence by $\left.U^i_\Lambda\right|_{\Sigma_\Lambda}<0$ we have $x^*\in T_\Lambda$. By mean value theorem, $$0\leq U^i_{\Lambda}(x_k)=2\frac{\partial u^i_\Lambda}{\partial x_1}(\xi_k)(\lambda_k-x_{k,1}),$$ where $\xi_k$ lies in the segment connecting $x_k$ and $x_k^{\lambda_k}$, which shows $\frac{\partial u^i_\Lambda}{\partial x_1}(\xi_k)\geq0$, hence $$\frac{\partial u^i_\Lambda}{\partial x_1}(x^*)\geq0,$$ By continuity of $u^i$, this is contradictory to $\left.\frac{\partial u^i}{\partial x_1}\right|_{T_\Lambda}=\left.-\frac12\frac{\partial U^i_\Lambda}{\partial x_1}\right|_{T_\Lambda}<0$. 
	
	\item  $\lim_{k\to\infty}|y_k|=+\infty$.  
	In this case, we can choose $k^*>0$ such that $\forall k>k^*$, $|y_k|>\widetilde{R}$, where $\widetilde{R}$ as in \cref{lemma:inf2}. Then for each $y_k, k>k^*$, the argument similar to the second part of \textbf{Step 1} will give a contradiction.
  \end{enumerate}
	
	\textbf{Step 3}: Now $\Lambda=\Lambda_0$, hence by continuity of $u^i$, $U^i_\Lambda\leq 0$, and $\forall\lambda<\Lambda$, by using \cref{lemma:SMPFinf} on $\Sigma_\lambda$, and by covering argument shows that $\frac{\partial u^i}{\partial x_1}<0$ on $\mleft\{x\in C_\infty\ \middle|\ x_1<\Lambda_0\mright\}$.
	
	For the second part assertion in the theorem, noting that \eqref{eq:EIC} also holds on $\Sigma_\Lambda$, hence if $\frac{\partial u^i}{\partial x_1}(\Lambda_0,x')=0$ for some $x\in \{x_1=\Lambda_0\}$, then by \cref{lemma:SMPFinf} again, $U^i_\Lambda\equiv0$ on $\Sigma_\Lambda$.
\end{prf6}

\vskip 0.2in
\section{Acknowledgements}

\bibliographystyle{abbrv}
\bibliography{ref}

\end{document}